[1994/06/01]

\documentclass[12pt,reqno,intlimits,english]{amsart}

\usepackage[T1]{fontenc}

\usepackage{calc}
\usepackage{ifthen}
\usepackage{babel}
\usepackage{xspace}
\usepackage{tabularray}
\usepackage{amsbsy}
\usepackage{cite}
\usepackage{epsfig,graphics,xcolor,calc,graphicx,pict2e,curve2e}

  \numberwithin{equation}{section}
  \newtheorem{theorem}{Theorem}[section]

\newtheorem{lemma}[theorem]{Lemma}
\newtheorem{proposition}[theorem]{Proposition}
\newtheorem{corollary}[theorem]{Corollary}

\theoremstyle{remark}
\newtheorem{Remark}[theorem]{Remark}
\newenvironment{remark}{\begin{Remark}}{\qed\end{Remark}}
\theoremstyle{definition}
\newtheorem{Definition}[theorem]{Definition}
\newenvironment{definition}{\begin{Definition}}{\qed\end{Definition}}

\makeatletter
\newcommand{\textupmd}[1]{\textup{\textmd{#1}}}
\newcommand{\unk}{u}
\newcommand{\unkd}{\unk_{\delta}}
\newcommand{\unkde}{\unk_{\eps,\delta}}
\newcommand{\vde}{v_{\eps,\delta}}
\newcommand{\initd}{\bar u}
\newcommand{\initdd}{\bar u_{\delta}}
\newcommand{\pder}[2]{\frac{\partial #1}{\partial #2}}

\DeclareMathOperator{\Lapl}{\Delta}
\DeclareMathOperator{\supp}{supp}
\newcommand{\@di}{\textupmd{d}}
\DeclareMathOperator{\dist}{dist}
\newcommand{\abs}[1]{\lvert#1\rvert}

\newcommand{\norm}[1]{\lVert#1\rVert}

\newcommand{\norma}[2]{\norm{#1}_{#2}}

\newcommand{\scpr}{\cdot}
\newcommand{\di}{\,\@di}
\newcommand{\grad}{\operatorname{\nabla}}
\DeclareMathOperator{\sign}{sign}
\newcommand{\Oset}{\varOmega}
\newcommand{\OsetT}{\Oset_{T}}
\newcommand{\Oint}{\Oset^{1}}
\newcommand{\OintT}{\Oset^{1}_{T}}
\newcommand{\Oout}{\Oset^{2}}
\newcommand{\OoutT}{\Oset^{2}_{T}}
\newcommand{\inter}{\varGamma}
\newcommand{\interT}{\varGamma_{T}}
\newcommand{\normal}{\nu}
\newcommand{\normint}{n}
\newcommand{\bdr}[1]{\partial #1}
\newcommand{\fpcd}{b_{\delta}}
\newcommand{\fpcde}{b_{\eps,\delta}}
\newcommand{\fpc}{b}

\newcommand{\funk}{v}

\newcommand{\funkd}{\funk_{\delta}}

\newcommand{\outtr}{\textupmd{out}}
\newcommand{\inttr}{\textupmd{int}}
\newcommand{\restr}[1]{{}_{\mid #1}}
\makeatother
\newcommand{\Y}{\mathcal{Y}}
\newcommand{\hole}{B}
\newcommand{\eps}{\varepsilon}
\newcommand{\Hper}{H^1_{\#}}
\newcommand{\R}{\mathbb{R}}
\newcommand{\ZZ}{\mathbb{Z}}
\newcommand{\Om}{\Oset}
\newcommand{\wto}{\rightharpoonup}
\newcommand{\const}{\gamma}
\newcommand{\Memb}{\varGamma_\eps}
\newcommand{\MembT}{\varGamma_{\eps,T}}
\newcommand{\unfolde}{\mathcal{T}_\eps}
\newcommand{\unfoldee}{\mathcal{T}_{\eps,\eta}}

\newcommand{\vi}{\upsilon}

\newcommand{\odu}{u}
\newcommand{\oddens}{\varphi}
\newcommand{\odfpc}{\beta}
\newcommand{\odfund}{\varGamma}
\newcommand{\odinit}{\alpha}
\newcommand{\oddir}{\varPhi}
\newcommand{\odv}{v}
\newcommand{\finmass}{m}
\newcommand{\msr}[2]{\mu_{#2}^{#1}}

\newcommand{\newatop}[2]{\genfrac{}{}{0pt}{0}{#1}{#2}}

\begin{document}

\title
{Upscaled equations for the 
Fokker--Planck diffusion through arrays of\\
permeable and of impermeable inclusions
}%
\author{Micol Amar}
\address{Dept. Basic and Applied Sciences for Engineering\\
  Sapienza University of Rome\\
  via A.Scarpa 16, 00161 Roma, Italy
}
\email{micol.amar@sbai.uniroma1.it}

\author{Daniele Andreucci}
\email{daniele.andreucci@sbai.uniroma1.it}

\author{Emilio N.M. Cirillo}
\email{emilio.cirillo@uniroma1.it}

\thanks{MA is member of Italian
G.N.A.M.P.A.--I.N.d.A.M. 
DA and ENMC are members of Italian
G.N.F.M.--I.N.d.A.M. The authors thank 
 the PRIN 2022 project
``Mathematical modelling of heterogeneous systems"
(code 2022MKB7MM, CUP B53D23009360006).
}

\maketitle

\begin{center}
\number \year-\ifthenelse{\number\month>9}{\relax}{0}\number\month-\ifthenelse{\number\day>9}{\relax}{0}\number\day
\end{center}

\begin{abstract}
We study the Fokker--Planck diffusion equation with 
diffusion coefficient depending periodically on the 
space variable. 
Inside a periodic array of inclusions
the diffusion coefficient is reduced by a factor called the 
diffusion magnitude. 
We find the upscaled equations obtained by taking 
both the degeneration and the homogenization limits in which 
the diffusion magnitude and the scale of the 
periodicity tends, respectively, to zero. 
Different behaviors, classified as pure diffusion, diffusion 
with mass deposition, and absence of diffusion,
are found depending on the order in which the two limits 
are taken and on the ratio between the size of the inclusions and 
the scale of the periodicity. 
\end{abstract}


\section{Introduction}\label{s:intro}
We consider the Fokker--Planck diffusion equation
\cite{Amar:Andreucci:Cirillo:2021,Andreucci:Colangeli:Cirillo:Gabrielli:2019}
for an inhomogeneous material whose diffusion properties
are encoded in a diffusion coefficient which oscillates rapidly
with respect to the space variable.

The Fokker--Planck equation is the evolution equation for the 
probability density function of diffusion stochastic processes
and is studied in several different contexts, ranging from 
statistical mechanics to information theory to economics to 
mean field games, see, e.g., \cite{Risken:1996,Furioli:Pulvirenti:Terraneo:Toscani:2017,Porretta:2015}.
Here, we are interested in the fact that it 
is also one of the two possible options 
\cite{Andreucci:Colangeli:Cirillo:Gabrielli:2019,MBCS2005,S2008,S1993},
together with the Fick equation,
to describe 
the
diffusion of particles in a medium with diffusion coefficient 
depending on the spatial coordinates. This behavior has been 
observed, for instance, when diffusing particles interact with a wall
\cite{L1984,HDL2006,LBLO2001,SD2005,LCW2016}, which is unavoidable when
the process takes place inside a confined region
\cite{CKMS2016,CKMSS2016,DmiKL2019,dCGdAM2015}.
Space dependent diffusion coefficients are also considered in some 
biological models to explain selection of ionic species 
\cite{AAB2017bis,ABC2014}.

Here, 
we assume that the material has a periodic microstructure
of \textit{characteristic length} $\varepsilon$.
Moreover, we introduce the parameter $\delta$ which
controls the magnitude of the diffusion coefficient
inside an $\varepsilon$--periodic array of permeable
inclusions whose size
is $\eta\varepsilon$. The parameters $\delta$ and $\eta$ will
be respectively called \textit{diffusion intensity} or
\textit{magnitude} inside the inclusions 
and \textit{relative size} of the inclusions.
The study is conducted in a bounded domain with 
zero flux (homogeneous Neumann) boundary conditions, so that, 
in absence of sources, the total mass would be conserved. 

We are interested to study the behavior of the system
in the \textit{degeneration} limit in which $\delta\to0$, namely, when
the mass diffusion inside the inclusions becomes negligible
so that inclusions become impenetrable.
In particular, we are interested in
finding upscaled equations in the \textit{homogenization} limit
$\varepsilon\to0$.

The degenerate problem has already been approached with
homogenization techniques
in the framework of the standard Fick diffusion equation,
see, e.g., 
\cite{Cabarrubias:Donato:2016}.
We stress that in that paper the point of view is different 
from the one that we adopt here, indeed, we obtain the degenerate 
problem as the limit for vanishing diffusion magnitude $\delta$ 
inside the inclusions, whereas in the previous paper inclusions 
were treated as holes of a perforated domain with prescribed 
Dirichlet
boundary conditions.
A thorough investigation of the Fick diffusion equation 
from our standpoint will be the topic of 
a future research.

We remark that,
starting from the pioneering paper 
\cite{Cioranescu:Murat:1982,Cioranescu:Paulin:1979}, in which the problem 
has been 
posed for an elliptic equation, many studies have 
appeared in the literature, mainly within the elliptic setup,
investigating this matter 
and showing that this topic 
has attracted the attention of mathematicians over more than four decades.
Without pretending to be exhaustive, we mention, for example, 
that 
the elliptic problem is 
considered again 
with homogeneous 
\cite{Allaire:1992,Allaire:Murat:1993}
and 
non--homogeneous 
\cite{Conca:Donato:1988,Hammouda:2008}. 
Neumann 
boundary conditions on the holes. 
We mention that 
in \cite{Cioranescu:Damlamian:Griso:Onofrei:2008} the similar 
problem of an elliptic equation for a Neumann sieve is considered. 
In \cite{Jurado:Diaz:2003} the parabolic problem with Dirichlet boundary 
conditions is attacked in a general abstract setup.
In the paper \cite{Cabarrubias:Donato:2016}, which can be considered 
the parabolic and hyperbolic version of 
\cite{Cioranescu:Murat:1982},
unfolding techniques have been 
applied to the wave and the Fick diffusion equation with 
homogeneous Dirichlet boundary conditions on the small hole boundary. 

Coming back to the 
present paper, here, 
we consider
the 
Fokker--Planck diffusion equation
and 
find the
limit equations in all the possible
cases obtained by tuning 
the inclusions size $\eta\varepsilon$
and 
taking the limits $\delta$ and $\eps$ to zero in the two orders
discussed below. 
The question we pose in this paper and the answer that we provide
have a natural mathematical interest. 
But this topic is 
also fascinating 
from the physical point of view, since
we find different macroscopic behaviors when 
the diffusion intensity in the inclusion, 
their size, and the 
characteristic scale of the overall periodicity
are changed.

Since in the paper we consider a rather large number of different
cases it is useful to list them in a sort of synoptic summary.
In the following $n$ will denote the space dimension
and we shall use the symbols $\approx$, $\ll$, and $\gg$
to distinguish among the different cases. The precise mathematical
meaning of those symbols will be provided in the sequel.
We shall mainly consider two different schemes to pass to the
degeneration $\delta\to0$ and to the 
homogenization $\varepsilon\to0$ limits.

\begin{figure}[t]
\begin{center}
\includegraphics[width=.25\textwidth]{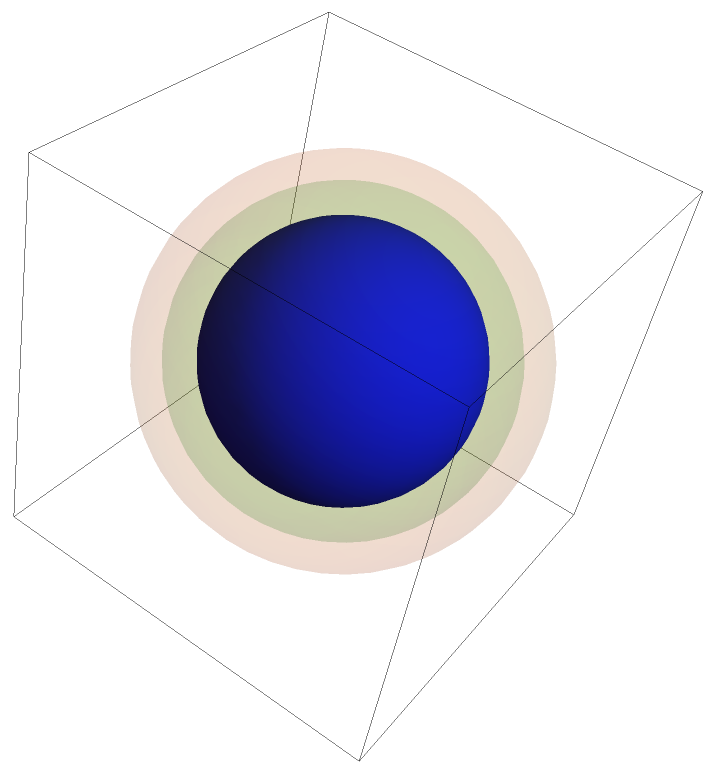}
\includegraphics[width=.25\textwidth]{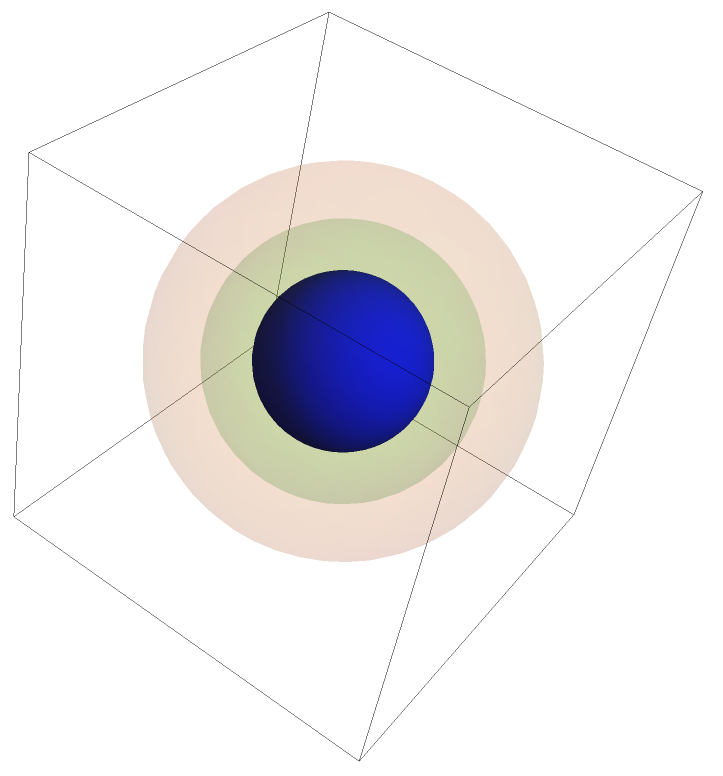}
\includegraphics[width=.25\textwidth]{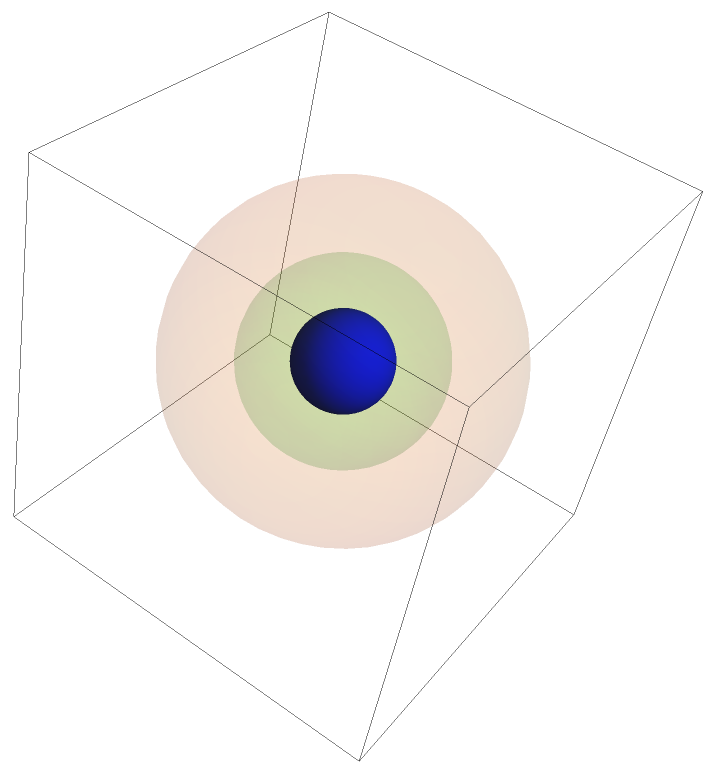}
\\
\includegraphics[width=.25\textwidth]{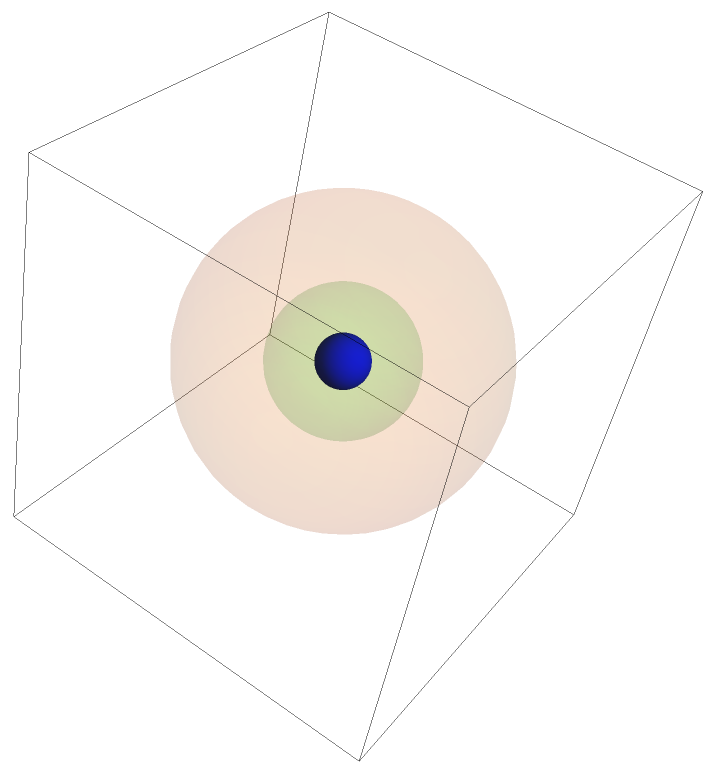}
\includegraphics[width=.25\textwidth]{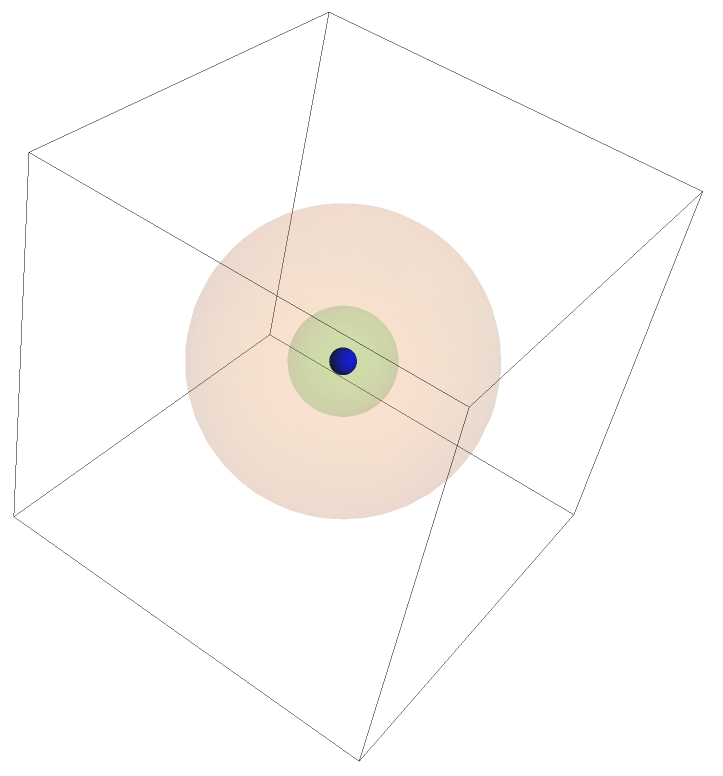}
\includegraphics[width=.25\textwidth]{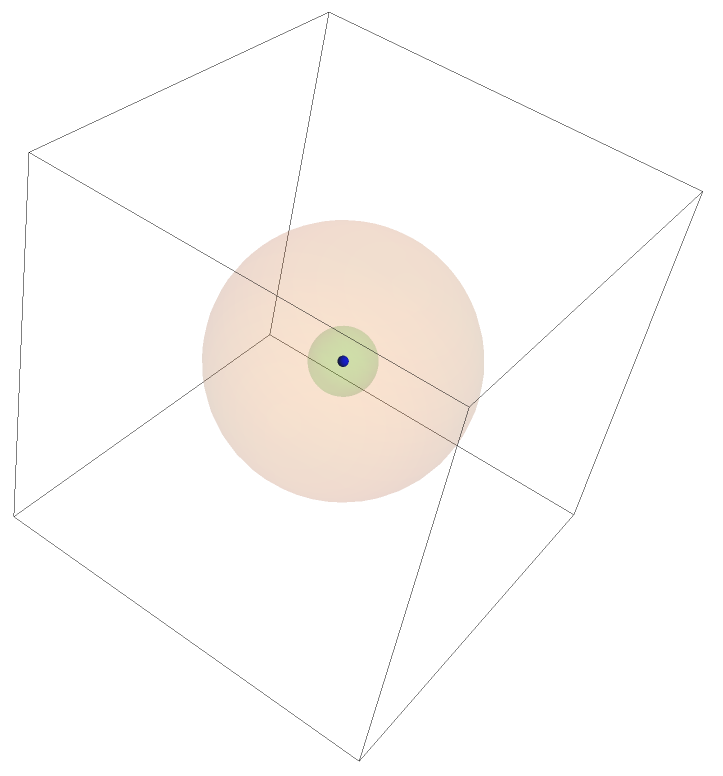}
\end{center}
\vskip .3 cm
\caption{Schematic representation of the geometry of cell and 
inclusions in the limiting supercritical, critical, and subcritical 
cases.
The elementary cell with side length $\eps$ rescaled
to unit is drawn in order to have comparable pictures.  
For $n=3$, 
the blue, the green, and the orange spheres 
represent the inclusion centered at the center of the cell 
respectively in the 
subcritical $\eta(\varepsilon)=(\eps^{2/(n-2)})^2$, 
critical $\eta(\varepsilon)=\eps^{2/(n-2)}$, 
and 
supercritical $\eta(\varepsilon)=(\eps^{2/(n-2)})^{1/4}$ 
cases.
In lexicographic order $\eps=0.9,0.8,0.7,0.6,0.5,0.4$.
}
\label{f:sc01}
\end{figure}

\textit{First asymptotic scheme.\/}
We first let $\delta\to0$ and then consider the homogenization limit
$\varepsilon\to0$.
To this end in Section~\ref{ss:limdelta} we find the
limit problem for $\delta\to0$ in Theorem~\ref{t:unkd_limit}
and call it the \textit{degenerate problem}.
Then, in Section~\ref{ss:homog_deg} we homogenize the
degenerate problem, but different choices for the behavior
of the relative inclusion size $\eta$
can be considered when $\varepsilon\to0$.
In Section~\ref{s:crit} we consider the so called
\textit{critical case}
$\eta\approx\varepsilon^{2/(n-2)}$, see \cite{Cioranescu:Murat:1982}.
The \textit{supercritical}
$\eta\gg\varepsilon^{2/(n-2)}$
and
the \textit{subcritical}
$\eta\ll\varepsilon^{2/(n-2)}$
cases are treated, respectively,
in Section~\ref{ss:superd} and \ref{ss:sottod}.

\textit{Second asymptotic scheme.\/}
We fix $\delta>0$ and consider the homogenization limit
$\varepsilon\to0$ in Sections~~\ref{ss:critical_ndeg}
and \ref{ss:superndeg}.
Such a limit depends on how
$\eta$ and $\varepsilon$ are related, so
the homogenization study 
is indeed divided into two parts:
in Section~\ref{ss:critical_ndeg}
we consider $\eta=\eta(\varepsilon)$ tending to
zero as $\varepsilon\to0$,
while
in Section~\ref{ss:superndeg}
we assume $\eta$ constant.
Then we pass to the limit $\delta\to0$ in
Section~\ref{s:deghom}.

We can summarize our results saying that, according to 
the dependence of the relative size $\eta$ on the cell size 
$\eps$, we find three possible behaviors for the upscaled 
equations: \textit{pure diffusion}, 
\textit{diffusion with mass deposition}, and \textit{absence of diffusion}.
More precisely, 
for what concerns the 
first asymptotic scheme, 
in Section~\ref{s:fp_nondeg},
the problems 
\eqref{eq:pde_2}--\eqref{eq:init_2} 
and
\eqref{eq:pde_1}--\eqref{eq:init_1} 
are found, respectively, 
outside and inside
the inclusions 
in the degeneration limit $\delta\to0$.
The former is a 
standard Fokker--Planck problem with homogeneous Dirichlet 
condition on the inclusions boundary.
The latter is an ordinary differential equation 
in time with a source term, which can be equivalently rewritten 
as equation \eqref{eq:add_F}.
It is to remark that the convergence to the solution of 
the limit problem as $\delta\to0$ inside the 
inclusions can be proven only on compact subdomains, 
since close to the inclusions boundary
a phenomenon of mass concentration takes place.
Indeed, in Theorem~\ref{t:add_limit} we show that 
the total mass in a vanishingly small strip adjacent from the inside 
to the inclusions boundary tends, in the degeneration limit, to 
the total mass flux arriving to the boundary from the 
exterior. 

When in Section~\ref{ss:homog_deg} we homogenize the degenerate 
equations derived in Section~\ref{s:fp_nondeg}, we find 
different upscaled systems depending on the way in which 
the relative size of the inclusions $\eta$ is scaled with 
respect to the cell size $\eps$. 
Referring to the nomenclature introduced above, 
in the subcritical regime inclusions have a poor effect and a standard
Fokker--Planck  
diffusion problem is found in Theorem~\ref{t:t2_supcrit}
with diffusion coefficient provided 
by a suitable cell average of the original coefficient. 
In the critical case, as in 
the pioneering paper
\cite{Cioranescu:Murat:1982}, 
inclusions are effective and yield a positive capacitary 
term in the diffusion equation of Theorem~\ref{t:t2}
accounting for mass deposition.
Finally, in the supercritical case, as shown in Theorem~\ref{t:superd},
inclusions
dominate and, provided the 
mass concentration phenomenon is correctly taken into account, 
the mass density converges to the solution of an ordinary 
differential equation in time as $\eps$ tends to zero. 

For what concerns the 
second asymptotic scheme, we have to consider two different cases. 
When the relative size $\eta$ tends to zero as $\eps\to0$, 
in Theorem~\ref{t:t22} we prove that the limit solution 
solves the standard Fokker--Planck diffusion 
problem \eqref{eq:m26ter} which does not depend 
on the degeneration parameter, so that no further analysis is needed. 
On the other hand, if the relative size is kept constant, say 
$\eta=1$, in Theorem~\ref{t:t23} we prove that, as $\eps\to0$, 
the limit solution solves the standard Fokker--Planck 
diffusion problem \eqref{eq:a46} with diffusion 
coefficient depending on the degeneration parameter $\delta$.
Moreover, as proven in Theorem~\ref{t:t10}, 
its solution, in the limit $\delta\to0$, tends to the solution of 
the ordinary differential equation \eqref{eq:add_F}.

We remark that the computation that we perform in
Section~\ref{ss:homog_deg}
is valid only in dimension $n\geq 3$, since we follow the ideas
in \cite{Cioranescu:Damlamian:Griso:Onofrei:2008} which are not
valid in smaller dimensions. On the contrary,
the results discussed 
in Sections~\ref{s:fp_nondeg} and \ref{ss:homog_ndeg} 
are in force for
any dimension $n\geq2$.

Finally, we note that, since in the two schemes the 
degeneration and the homogenization limits are taken in 
reversed order, it is natural to compare these
results each other and look for possible commutation properties.
We refer to Remark~\ref{r:r1} for a thorough discussion, 
but, here, 
we anticipate that 
the two strategies commute
when 
$\eta=1$,
whereas 
when $\eta\to0$ as $\eps\to0$ 
they commute in the
subcritical case, 
while 
in the 
critical 
and 
supercritical cases
they do not.

In view of the variety of these results, a natural question arises about the 
behavior of the model when the degeneration and the homogenization limits 
are taken simultaneously, namely, when the parameter $\delta$ 
is considered a vanishing function of $\eps$. 
Preliminary results suggest that this can be a promising study and, thus,
it will be the topic of future research.

The paper is organized as follows. 
In Section~\ref{s:problem} we introduce the model. 
In Section~\ref{s:fp_nondeg} we discuss the 
degeneration $\delta\to0$ limit.
Section~\ref{s:unfold} is devoted to a short review of 
the unfolding approach to homogenization. 
In Sections~\ref{ss:homog_deg}
and \ref{ss:homog_ndeg} we study, respectively, the 
homogenization limit of the degenerate and the non--degenerate problems.
In Section~\ref{s:od} we provide an explicit solution 
of the problem under investigation in the one--dimensional case 
showing that, if one considered diffusion coefficient depending 
on time, globally bounded solutions could not exist.
Finally, in Section~\ref{s:concl}
we summarize our conclusions.

\section{The problem}
\label{s:problem}
Let $\Oset\subset\R^n$ be a smooth bounded open set.
Let $\eps>0$ be a small parameter denoting the length scale of the periodic microstructure.
Let us consider the tiling of $\mathbb{R}^n$
given by the boxes $\varepsilon(\xi+\Y)$,
with $\xi\in\mathbb{Z}^n$ and $\Y=(-1/2,1/2)^n$.
We denote by $[r]$ the integer part of $r\in\mathbb{R}$ with respect to 
the reference cell $(-1/2,1/2)$
(i.e., $[r]=k\in \ZZ$ if and only 
if $r\in[k-1/2,k+1/2)$) and, similarly, we denote by
$\{r\}=r-[r]$, i.e., 
the fractional part of $r$ with respect to $(-1/2,1/2)$. Moreover,
for $x\in\mathbb{R}^n$, we define the vector with integer components
$[x]_\Y=([x_1],\dots,[x_n])$.
We refer to Fig.~\ref{f:def} for a schematic
representation of the geometric setup.

We set
\begin{equation}
\label{fold000}
\begin{aligned}
& \Xi_\varepsilon
=
\{\xi\in\mathbb{Z}^n:\,
  \varepsilon(\xi+\Y)\subset\Omega\},
\
\hat\Omega_\varepsilon
=
\text{interior}
\Big\{
      \bigcup_{\xi\in\Xi_\varepsilon}
           \varepsilon(\xi+\overline{\Y})
\Big\}
,
\\
&\hphantom{jhkfgjdhsaJ,KHGFD}
\Lambda_{\varepsilon}
=
\Omega\setminus\hat\Omega_\varepsilon
\;.
\end{aligned}
\end{equation}
We introduce also the scaled cell containing the point $x$ as
$$
\Y_\eps(x)=\eps\left(\left[\frac{x}{\eps}\right]_\Y+\Y\right)\,.
$$
In the sequel, we will assume that $\Oset$ contains an $\eps\Y$-periodic 
array of smooth small holes of size $\eta\eps$ ($1\ge \eta>0$ possibly 
depending on $\eps$).
More precisely, if the reference inclusion (also called hole, as 
in the previous literature)  $\hole\subset\subset \Y$ is a 
given connected regular open set, we denote by $\Y^*_\eta
=\Y\setminus\eta \overline \hole$ and define $\Om^*_{\eps,\eta}$ as
\begin{equation}\label{eq:m2}
\begin{aligned}
\Om^*_{\eps,\eta} & =\hbox{interior}\left\{\bigcup_{\xi\in\Xi_\varepsilon}
           \varepsilon(\xi+\overline\Y^*_\eta)\right\}
\\
&= \left\{x\in\hat\Om_\eps,\ \hbox{such that\ } \frac{x}{\eps}-\left[\frac{x}{\eps}\right]_\Y\in \Y^*_\eta\right\}.
\end{aligned}
\end{equation}
We also denote by $\Oint_\eps=\hat\Om_\eps\setminus\overline{\Om^*_{\eps,\eta}}
=\cup_{\xi\in\Xi_\varepsilon} \varepsilon(\xi+\eta\hole)$ and $\Oout_\eps=\Lambda_\eps\cup\Om^*_{\eps,\eta}=\Om\setminus\overline{\Oint_{\eps}}$, respectively,
and we assume that, for every $\eps>0$, they are smooth sets.
For the sake of simplicity, we also denote by $\Memb=\partial\Oout_\eps\setminus\partial\Om=\partial\Oint_\eps$,
so that $\Oset=\Oint_\eps\cup\Oout_\eps\cup\Memb$; that is $\Oint_{\eps}$ 
is the interior of the inclusions and $\Oout_{\eps}$ is the outer domain. 

\setlength{\unitlength}{1.3pt}
\begin{figure}[h]
\begin{picture}(400,100)(-25,20)
\thinlines
\multiput(0,0)(0,20){6}{\Dline(0,0)(100,0){3}}
\multiput(0,0)(20,0){6}{\Dline(0,0)(0,100){3}}
\thicklines
\put(-10,50){\vector(1,0){120}}
\put(50,-10){\vector(0,1){120}}
\thinlines
\put(67,44){${\scriptstyle k_1}$}
\put(70,50){\line(0,1){40}}
\put(42,86){${\scriptstyle k_2}$}
\put(50,90){\line(1,0){20}}
\put(72,87){\circle*{2}}
\put(74,87){${\scriptstyle x}$}
\multiput(5,14)(0,20){5}{\multiput(0,0)(20,0){5}{\begin{color}{gray}\circle*{6}\end{color}}}
\multiput(140,5)(0,10){10}{\Dline(0,0)(90,0){2}}
\multiput(140,5)(10,0){10}{\Dline(0,0)(0,90){2}}
\thicklines
\put(130,50){\vector(1,0){110}}
\put(185,-5){\vector(0,1){110}}
\thinlines
\put(185,50){\circle{84}}
\linethickness{0.7mm}
\put(150,35){\line(0,1){30}}
\put(150,65){\line(1,0){10}}
\put(160,65){\line(0,1){10}}
\put(160,75){\line(1,0){10}}
\put(170,75){\line(0,1){10}}
\put(170,85){\line(1,0){30}}
\put(200,85){\line(0,-1){10}}
\put(200,75){\line(1,0){10}}
\put(210,75){\line(0,-1){10}}
\put(210,65){\line(1,0){10}}
\put(220,65){\line(0,-1){30}}
\put(220,35){\line(-1,0){10}}
\put(210,35){\line(0,-1){10}}
\put(210,25){\line(-1,0){10}}
\put(200,25){\line(0,-1){10}}
\put(200,15){\line(-1,0){30}}
\put(170,15){\line(0,1){10}}
\put(170,25){\line(-1,0){10}}
\put(160,25){\line(0,1){10}}
\put(160,35){\line(-1,0){10}}
\multiput(142.5,12)(0,10){9}{\multiput(0,0)(10,0){9}{\begin{color}{gray}\circle*{3}\end{color}}}
\end{picture}
\vskip 2. cm
\caption{Schematic description of the
geometry of the model in dimension $n=2$ and
some related notions.
The gray dots represent the inclusions.
Left: tiling and definition of
integer part $[x]_\mathcal{Y}=(k_1,k_2)$.
On the right the lattice is rescaled with $\varepsilon=1/2$ and
$\eta=1$:
the big circle represents the open
set $\Omega$ and the region with solid boundary is the set
$\bigcup_{\xi\in\Xi_\varepsilon}\varepsilon(\xi+\overline{\mathcal{Y}})$.
}
\label{f:def}
\end{figure}
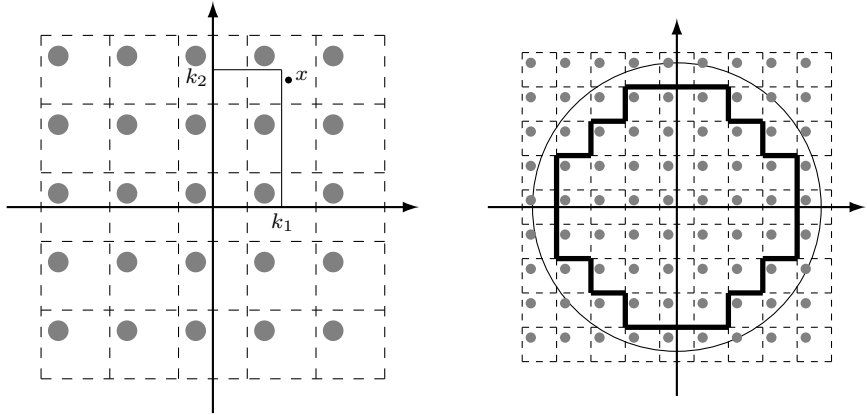

\noindent
Finally, for any set $\mathcal{G}\subset \R^n$, we denote $\mathcal{G}_T=\mathcal{G}\times(0,T)$.

\medskip
Let us consider the problem

\begin{alignat}{2}
  \label{eq:pde_d}
  \pder{\unkde}{t}
  -
  \Lapl(\fpcde \unkde)
  &=
  f
  \,,
  &\qquad&
  \text{in $\OsetT$;}
  \\
  \label{eq:bdr_d}
  \pder{(\fpcde\unkde)}{\normal}
  &=
  0
  \,,
  &\qquad&
  \text{on $\bdr{\Oset}\times(0,T)$;}
  \\
  \label{eq:init_d}
  \unkde(x,0)
  &=
  \initd(x)
  \,,
  &\qquad&
  \text{in $\Oset$.}
\end{alignat}
Here $\normal$ denotes the outer normal to $\bdr{\Oset}$, $f\in L^{2}(\OsetT)$, $\initd\in L^2(\Om)$. For $0<\delta\le 1$ and $\eps>0$
\begin{equation}
  \label{eq:m3}
  \fpcde(x)
  =
  \left\{
    \begin{alignedat}{2}
      &\delta\fpc_\eps(x)
      \,,
      &\qquad&
      x\in \Oint_\eps
      \,;
      \\
      &\fpc_\eps(x)
      \,,
      &\qquad&
      x\in \Oout_\eps
      \,,
    \end{alignedat}
    \right.
\end{equation}
with $\fpc_\eps(x)=\fpc(x,x/\eps)$, where $\fpc$ is a Carath\'{e}odory
function belonging to $L^\infty(\Om\times\Y)$,
which is $\Y$-periodic with respect to the second variable
and satisfies
\begin{equation}
  \label{eq:m7}
  \fpc(x,y)
  \ge
  C
  >
  0
  \,,
  \qquad
  (x,y)\in\Oset\times\Y
  \,,
\end{equation}
for a suitable constant $C>0$.

Notice that, as explained in the Introduction,
the appearance of two small parameters $\eps$ and $\delta$ in
problem \eqref{eq:pde_d}--\eqref{eq:init_d} leads us to consider, and
compare, the behavior of the problem when we let first $\delta\to 0$
and then $\eps\to 0$ or vice versa.
These situations will be analyzed in the
following sections.

\begin{definition}
  \label{d:weaksol_d}
  A weak solution to problem \eqref{eq:pde_d}--\eqref{eq:init_d} 
is a function $\unkde\in H^{1}(t_{0},T;L^{2}(\Oset))\cap \mathcal{C}([0,T];L^{2}(\Oset))$ for all $0<t_{0}<T$, such that $\fpcde\unkde\in L^{2}(0,T;H^{1}(\Oset))$ and
  \begin{equation}
    \label{eq:weaksol_d_a}
    \int_{\OsetT}
    \Big[
    \pder{\unkde}{t}
    \varphi
    +
    \grad(\fpcde\unkde)
    \grad \varphi
    \Big]
    \di x
    \di t
    =
    \int_{\OsetT}
    f\varphi
    \di x
    \di t
    \,,
  \end{equation}
  for all $\varphi\in L^{2}(0,T;H^{1}(\Oset))$, with support bounded away from $t=0$. In addition, we require $\unkde=\initd$ at time $t=0$ in the $L^{2}$ sense.
\end{definition}

We may give to \eqref{eq:weaksol_d_a} the equivalent alternative formulation
\begin{multline}
  \label{eq:weak_form22}
  -
  \int_{\OsetT}
  \unkde \phi_t
  \di x\di t
  +
  \int_{\OsetT}
  \grad (\fpcde\unkde)
  \grad\phi
  \di x\di t
  \\
  =
  \int_{\OsetT}
  f
  \phi
  \di x\di t
  +
  \int_{\Oset}
  \initd(x)
  \phi(x,0)
  \di x
  \,,
\end{multline}
for all test functions $\phi\in H^1(\OsetT)$, with $\phi(x,T)=0$ in $\Oset$, in the sense of traces.
For example we may choose for any $\varphi\in \mathcal{C}^{1}(\overline{\OsetT})$
\begin{equation}
  \label{eq:weaktest}
  \phi(x,t)
  =
  \int_{t}^{T}
  \varphi(x,\tau)
  \di\tau
  \,.
\end{equation}

\subsection{Energy and $L^{1}$ estimates}
\label{s:energy}
Here we collect some results which are used
throughout the paper.

An immediate consequence of \eqref{eq:weaksol_d_a} is the balance
\begin{equation}
  \label{eq:mass_balance}
  \int_{\Oset}
  \unkde(x,t)
  \di x
  =
  \int_{\Oset}
  \initd(x)
  \di x
  +
  \int_{0}^{t}
  \int_{\Oset}
  f(x,\tau)
  \di x
  \di \tau
  \,,
  \quad
  0<t<T
  \,.
\end{equation}

We infer by \eqref{eq:weaksol_d_a} and routine arguments the energy estimate
\begin{multline}
  \label{eq:energy_d_n}
  \sup_{0<t<T}
  \int_{\Oset}
  \fpcde
  \unkde(t)^{2}
  \di x
  +
  \int_{\OsetT}
  \abs{\grad(\fpcde\unkde)}^{2}
  \di x
  \di t
  \\
  \le
  \gamma(T)
  \Big(
  \int_{\Oset}
  \fpcde
  \initd^{2}
  \di x
  +
  \int_{\OsetT}
  \fpcde
  f^{2}
  \di x
  \di t
  \Big)
  \,,
\end{multline}
where, here and in the following, $\gamma$ is a generic positive 
real number not depending on $\varepsilon$, $\delta$, and $\eta$.

We also have, from 
\eqref{eq:weaksol_d_a} and \eqref{eq:energy_d_n},
by choosing as test function the product of 
$b_{\varepsilon,\delta}\partial u_{\varepsilon,\delta}/\partial t$ times 
a continuous function of time constant for $t\ge t_0$ and zero in $0$,
that
for all $0<t_{0}<T$
\begin{multline}
  \label{eq:energyt0_d_n}
  \int_{t_{0}}^{T}
  \int_{\Oset}
  \fpcde
  \Big(
  \pder{\unkde}{t}
  \Big)^{2}
  \di x
  \di t
  +
  \sup_{t_{0}<t<T}
  \int_{\Oset}
  \abs{\grad(\fpcde\unkde(t))}^{2}
  \di x
  \\
  \le
  \gamma(T)
  t_{0}^{-1}
  \Big(
  \int_{\Oset}
  \fpcde
  \initd^{2}
  \di x
  +
  \int_{\OsetT}
  \fpcde
  f^{2}
  \di x
  \di t
  \Big)
  \,.
\end{multline}
Such estimates may be used to prove existence of a solution in the sense of Definition~\ref{d:weaksol_d}, via approximation with smoothed problems, and also its uniqueness, since the problem is linear. Alternatively, uniqueness follows from our next result. 

\begin{lemma}[Conservation of mass]
  \label{l:mass_con}
  We have for all $t\in(0,T)$
  \begin{equation}
    \label{eq:mass_con_n}
    \int_{\Oset}
    \abs{\unkde(x,t)}
    \di x
    \le
    \int_{\Oset}
    \abs{\initd(x)}
    \di x
    +
    \int_{0}^{t}
    \int_{\Oset}
    \abs{f}
    \di x
    \di \tau
    \,.
  \end{equation}
  If $\initd\ge 0$, $f\ge 0$ then $\unkde\ge 0$ and, if in addition $f=0$,
  \begin{equation}
    \label{eq:mass_con_nn}
    \norma{\unkde(t)}{L^{1}(\Oset)}
    =
    \norma{\initd}{L^{1}(\Oset)}
    \,,
    \qquad
    0<t<T
    \,.
  \end{equation}
\end{lemma}

\begin{proof}
  Let $\sign_{\sigma}$ be a smoothed increasing version of the sign
  function, with $\sign_{\sigma}(0)=0$, 
  converging everywhere to $\sign$ as $\sigma\to  0+$, and select as a
  testing function in \eqref{eq:weaksol_d_a}
  $\sign_{\sigma}(\fpcde\unkde)\chi_{(t_{0},t)}(\tau)$. We get
  \begin{multline*}
    \int_{t_{0}}^{t}
    \int_{\Oset}
    \Big[
    \pder{\unkde}{\tau}
    \sign_{\sigma}(\fpcde\unkde)
    +
    \sign_{\sigma}'(\fpcde\unkde)
    \abs{\grad(\fpcde\unkde)}^{2}
    \Big]
    \di x
    \di \tau
    \\
    \le
    \int_{t_{0}}^{t}
    \int_{\Oset}
    \abs{f}
    \di x
    \di \tau
    \,.
  \end{multline*}
  Next we drop the non-negative term on the left hand side, then we let $\sigma\to 0+$ and note that $\sign(\fpcde\unkde)=\sign(\unkde)$; thus we obtain
  \begin{multline*}
    \int_{t_{0}}^{t}
    \int_{\Oset}
    \abs{f}
    \di x
    \di \tau
    \ge
    \int_{t_{0}}^{t}
    \int_{\Oset}
    \pder{\unkde}{\tau}
    \sign(\unkde)
    \di x
    \di \tau
    \\
    =
    \int_{t_{0}}^{t}
    \int_{\Oset}
    \pder{\abs{\unkde}}{\tau}
    \di x
    \di \tau
    =
    \int_{\Oset}
    \abs{\unkde(x,t)}
    \di x
    -
    \int_{\Oset}
    \abs{\unkde(x,t_{0})}
    \di x
    \,.
  \end{multline*}
  On letting $t_{0}\to 0$ we arrive at \eqref{eq:mass_con_n}.

  The positivity result follows similarly, by replacing $\sign_{\sigma}(s)$ with  $\sign_{\sigma}(s)\chi_{(-\infty,0)}(s)$; then \eqref{eq:mass_con_nn} follows from \eqref{eq:mass_balance}.
\end{proof}

\subsection{An auxiliary formulation.}
\label{s:vfor}
If we set $\vde=\fpcde\unkde$ we obtain for this new unknown the problem
\begin{alignat}{2}
  \label{eq:pde_dv}
  \frac{1}{\fpcde}
  \pder{\vde}{t}
  -
  \Lapl \vde
  &=
  f
  \,,
  &\qquad&
  \text{in $\OsetT$;}
  \\
  \label{eq:bdr_dv}
  \pder{\vde}{\normal}
  &=
  0
  \,,
  &\qquad&
  \text{on $\bdr{\Oset}\times(0,T)$;}
  \\
  \label{eq:init_dv}
  \vde(x,0)
  &=
  \fpcde(x)\initd(x)
  \,,
  &\qquad&
  \text{in $\Oset$.}
\end{alignat}
The weak formulations follow obviously from the ones in Definition~\ref{d:weaksol_d} and in \eqref{eq:weak_form22}; let us write explicitly the latter form as
\begin{multline}
  \label{eq:weak_form12}
  -
  \int_{\OsetT}
  \frac{\vde}{\fpcde}
  \phi_t
  \di x\di t
  +
  \int_{\OsetT}
  \grad \vde
  \grad\phi
  \di x\di t
  \\
  =
  \int_{\OsetT}
  f\phi
  \di x\di t
  +
  \int_{\Oset}
  \initd(x)
  \phi(x,0)
  \di x\,,
\end{multline}
for all test functions $\phi\in H^1(\OsetT)$, with $\phi(x,T)=0$ in $\Oset$, in the sense of traces.

From the estimates \eqref{eq:energy_d_n} and
\eqref{eq:energyt0_d_n} we obtain for all $0<t_{0}<T$
\begin{equation}
  \label{eq:energy_v}
  \sup_{0<t<T}
  \int_{\Oset}
  \frac{\vde(t)^{2}}{\fpcde}
  \di x
  +
  \int_{\OsetT}
  \abs{\grad\vde}^{2}
  \di x
  \di t
  +
  t_{0}
  \int_{t_{0}}^{T}
  \int_{\Oset}
  \frac{1}{\fpcde}
  \Big(
  \pder{\vde}{t}
  \Big)^{2}
  \di x
  \di t
  \le
  \gamma
  \,,
\end{equation}
with $\gamma$ as above and independent of $t_{0}$.

\begin{remark}
\label{r:vcomp_ext}
As a consequence of \eqref{eq:energy_v}, the sequence $\vde$ 
is compact in $L^{2}(\OsetT)$ and we have, up to subsequences
for $\varepsilon\to0$ and $\delta\to0$,
  \begin{alignat}{2}
    \label{eq:vcomp_ext_v}
    \vde&\to v_{0}
    \,,
    &\qquad&
    \text{strongly in $L^{2}(\OsetT)$;}
    \\
    \label{eq:vcomp_ext_grad}
    \grad\vde &\wto \grad v_{0}
    \,,
    &\qquad&
    \text{weakly in $L^{2}(\OsetT)$;}
    \\
    \pder{\vde}{t}&\wto\pder{v_{0}}{t}
    \,,
    &\qquad&
    \text{weakly in $L^{2}(\Oset\times(t_{0},T))$, for all $t_{0}>0$.}
  \end{alignat}
\end{remark}

\begin{remark}[$L^{\infty}$ bounds for $\fpcde$ independent of time]
  \label{r:supbound}
  It is possible to prove (see, \cite{Andreucci:Bellaveglia:Cirillo:2014}) 
  that, if for example $f=0$ and $\initd\ge 0$,
  \begin{equation}
    \label{eq:supbound_a}
    0
    \le
    \unkde(x,t)
    \le
    \frac{\sup_{\Oset}\fpcde}{\inf_{\Oset}\fpcde}
    \,
    \sup_{\Oset}\initd
    \,,
    \qquad
    (x,t)\in\OsetT
    \,.
  \end{equation}
This result relies on the independence of $\fpcde$ 
from time: see Section~\ref{s:od} for further comments 
and a counterexample motivating our choice $\fpcde=\fpcde(x)$ independent
on time in this paper.

However, we note that the definitions of weak solutions 
\eqref{eq:weaksol_d_a}, \eqref{eq:weak_form22}, and \eqref{eq:weak_form12}
would be still valid even if $\fpcde$ depended on time.
\end{remark}

\section{The degeneration limit of the Fokker--Planck problem}
\label{s:fp_nondeg}

In this section, we will assume $\eps$ fixed (i.e., $\eps=1$) 
and study the behavior of the solution with respect to $\delta \to 0$.
For this reason, we will omit the subscript index $\eps$ 
and replace $\Oint_\eps$, $\Oout_\eps$, and $\Gamma_\varepsilon$
with $\Oint$, $\Oout$, and $\Gamma$, respectively.
Moreover, we will use the superscripts ``in'' and ``out''
to denote restrictions to $\Oint$ and $\Oout$.

Since, 
the spatial periodicity of cavities is not important in this section, 
we can just assume that the smooth bounded open sets 
$\Oint$ and $\Oout$ satisfy 
the following assumptions:
$\Oint\subset\overline{\Oint}\subset\Oset$
and 
$\Oout=\Oset\setminus\overline{\Oint}$.
We also let $\unkd=\unkde$, $\funkd=\vde$, and $\fpcd=\fpcde$.

We stress that 
we will rely on some non-standard energy estimates where the tracking 
of the behavior in $\delta$ is rather delicate (see especially 
Lemma~\ref{l:add_expand}). 

First of all we note that 
standard arguments and the assumed regularity of $\fpcd\unkd$ 
imply that a weak solution to \eqref{eq:pde_d}--\eqref{eq:init_d},
which is smooth enough in $\Oint$ and in $\Oout$, satisfies
\begin{alignat}{2}
  \label{eq:pde_d1}
  \pder{\unkd}{t}
  -
  \delta
  \Lapl(\fpc\unkd)
  &=
  f
  \,,
  &\qquad&
  \text{in $\OintT$;}
  \\
  \label{eq:pde_d2}
  \pder{\unkd}{t}
  -
  \Lapl(\fpc\unkd)
  &=
  f
  \,,
  &\qquad&
  \text{in $\OoutT$;}
  \\
  \label{eq:bdr_d2}
  \pder{(\fpc\unkd)}{\normal}
  &=
  0
  \,,
  &\qquad&
  \text{on $\bdr{\Oset}\times(0,T)$;}
  \\
  \label{eq:jump_d}
  [\fpcd \unkd]
  &=
  0
  \,,
  &\qquad&
  \text{on $\interT$;}
  \\
  \label{eq:jumpflux_d}
  [\grad(\fpcd\unkd)\scpr\normint]
  &=
  0
  \,,
  &\qquad&
  \text{on $\interT$;}
  \\
  \label{eq:init_d12}
  \unkd(x,0)
  &=
  \initd(x)
  \,,
  &\qquad&
  \text{in $\Oset$.}
\end{alignat}
Recall that
$\inter=\bdr{\Oint}\cap\bdr{\Oout}$ and $\normint$ is the normal
to $\inter$ pointing into $\Oout$. Let us remark that
\eqref{eq:jump_d} follows from the fact that $\fpcd\unkd$ is a Sobolev
function and \eqref{eq:jumpflux_d} is a standard consequence of the
differential equation \eqref{eq:pde_d} understood in a distributional sense.
Note, also, that \eqref{eq:jump_d} implies that $\unkd$ is not continuous 
across the interface $\Gamma$.

\subsection{The limit degenerate problem}

The point of the following estimate is that it is independent of $\delta$ (excepting the factor in the second integral of \eqref{eq:energy_dint_n}).

\label{ss:limdelta}
\begin{lemma}
  \label{l:energy_dint}
  For all $\varphi\in \mathcal{C}^{1}_{0}(\Oint)$, $0\le \varphi\le 1$, we have
  \begin{multline}
    \label{eq:energy_dint_n}
    \sup_{0<t<T}
    \int_{\Oint}
    \fpc
    \unkd(t)^{2}
    \varphi^{2}
    \di x
    +
    \delta
    \int_{0}^{t}
    \int_{\Oint}
    \abs{\grad(\fpc\unkd)}^{2}
    \varphi^{2}
    \di x
    \di \tau
    \\
    \le
    \gamma
    (1+\norma{\grad\varphi}{\infty}^{2})
    \Big(
    \int_{\Oset}
    \initd^{2}
    \di x
    +
    \int_{\OsetT}
    f^{2}
    \di x
    \di t
    \Big)
    \,.
  \end{multline}
  Here $\gamma$ depends on $T$.
\end{lemma}

\begin{proof}
The proof is based on standard arguments that we report for 
the reader's convenience. Indeed, 
select $\varphi^{2}\fpc\unkd\chi_{[t_{0},t]}(\tau)$ as a testing function in \eqref{eq:weaksol_d_a}, integrate by parts and let $t_{0}\to 0+$ to  get
  \begin{multline*}
    \!
    \!
    \!
    \frac{1}{2}
    \int_{\Oint}
    \fpc
    \unkd(t)^{2}
    \varphi^{2}
    \di x
    +
    \delta
    \int_{0}^{t}
    \int_{\Oint}
    \abs{\grad(\fpc\unkd)}^{2}
    \varphi^{2}
    \di x
    \di \tau
    \\
    \;
    =
    \frac{1}{2}
    \int_{\Oint}
    \fpc
    \initd^{2}
    \varphi^{2}
    \di x
    +
    2
    \delta
    \int_{0}^{t}
    \int_{\Oint}
    \grad(\fpc\unkd)
    \grad\varphi
    \varphi
    \unkd
    \fpc
    \di x
    \di \tau
    +
    \int_{0}^{t}
    \int_{\Oint}
    f
    \varphi^{2}
    \unkd
    \fpc
    \di x
    \di \tau
    \,.
  \end{multline*}
  The sum of the last two integrals is bounded from above by
  \begin{multline*}
    \frac{\delta}{2}
    \int_{0}^{t}
    \int_{\Oint}
    \abs{\grad(\fpc\unkd)}^{2}
    \varphi^{2}
    \di x
    \di \tau
    +
    2 \delta
    \int_{0}^{t}
    \int_{\Oint}
    \abs{\grad\varphi}^{2}
    \unkd^{2}
    \fpc^{2}
    \di x
    \di \tau
    \\
    +
    \int_{0}^{t}
    \int_{\Oint}
    f^{2}
    \varphi^{2}
    \di x
    \di \tau
    +
    \int_{0}^{t}
    \int_{\Oint}
    \unkd^{2}
    \varphi^{2}
    \fpc^{2}
    \di x
    \di \tau
    \,.
  \end{multline*}
  By taking into account \eqref{eq:energy_d_n} we infer that
  \begin{multline*}
    \int_{\Oint}
    \unkd(t)^{2}
    \varphi^{2}
    \fpc
    \di x
    +
    \delta
    \int_{0}^{t}
    \int_{\Oint}
    \abs{\grad(\fpc\unkd)}^{2}
    \varphi^{2}
    \di x
    \di \tau
    \\
    \le
    \gamma
    (1+\norma{\grad\varphi}{\infty}^{2})
    \Big(
    \int_{\Oset}
    \initd^{2}
    \di x
    +
    \int_{0}^{t}
    \int_{\Oset}
    f^{2}
    \di x
    \di t
    \Big)
    +
    \gamma
    \int_{0}^{t}
    \int_{\Oint}
    \unkd^{2}
    \varphi^{2}
    \fpc
    \di x
    \di t
    \,.
  \end{multline*}
  The claim follows after an application of Gronwall's lemma.
\end{proof}

Next we show that as $\delta\to 0$, $\unkd$ converges, in the respective spatial domains, to the solutions of the two following problems. The problem in the outer domain $\Oout$ is
\begin{alignat}{2}
  \label{eq:pde_2}
  \pder{\unk}{t}
  -
  \Lapl(\fpc\unk)
  &=
  f
  \,,
  &\qquad&
  \text{in $\OoutT$;}
  \\
  \label{eq:bdr_2}
  \pder{(\fpc\unk)}{\normal}
  &=
  0
  \,,
  &\qquad&
  \text{on $\bdr{\Oset}\times(0,T)$;}
  \\
  \label{eq:jump}
  \fpc\unk
  &=
  0
  \,,
  &\qquad&
  \text{on $\interT$;}
  \\
  \label{eq:init_2}
  \unk(x,0)
  &=
  \initd(x)
  \,,
  &\qquad&
  \text{in $\Oout$.}
\end{alignat}
Problem \eqref{eq:pde_2}--\eqref{eq:init_2} has the standard weak 
formulation: Find $\unk\in L^{2}(\OsetT)$, 
with $\fpc\unk\in L^{2}(0,T;H^{1}(\Oset))$ and satisfying \eqref{eq:jump}, 
such that
\begin{equation}
  \label{eq:weaksol_out}
  \int_{\OoutT}
  \Big[
  -\unk
  \pder{\varphi}{t}
  +
  \grad(\fpc\unk)
  \grad \varphi
  \Big]
  \di x
  \di t
  =
  \int_{\Oout}
  \initd
  \varphi(0)
  \di x
  +
  \int_{\OoutT}
  f\varphi
  \di x
  \di t
  \,,
\end{equation}
for all $\varphi\in H^{1}(\OoutT)$, with $\varphi=0$ on $\inter$ and at $t=T$.

The problem in the interior domain $\Oint$ is: 
Find $\unk$ such that $\unk$, 
$\partial \unk/\partial t\in L^{2}(\Oint_{T})$ and
\begin{alignat}{2}
  \label{eq:pde_1}
  \pder{\unk}{t}
  &=
  f
  \,,
  &\qquad&
  \text{in $L^{2}(\Oint_{T})$;}
  \\
  \label{eq:init_1}
  \unk(x,0)
  &=
  \initd(x)
  \,,
  &\qquad&
  \text{in $\Oint$, in the sense of traces.}
\end{alignat}
In fact, it is easy to prove that \eqref{eq:pde_1}--\eqref{eq:init_1}
can be written equivalently as $u=F$ in $\Oint$ where
\begin{equation}
  \label{eq:add_F}
  F(x,t)
  =
  \initd(x)
  +
  \int_{0}^{t}
  f(x,\tau)
  \di\tau
  \,,
  \qquad
  (x,t)\in\OsetT
  \,.
\end{equation}

\begin{theorem}
  \label{t:unkd_limit}
  As $\delta\to 0$, for every fixed $t_{0}>0$,
  \begin{equation}
    \label{eq:unkd_out_strong}
    \begin{aligned}
   \unkd
 &
    \to
    \unk
    \,,
    \qquad
    & \textup{strongly in $L^{2}(\OoutT)$;}
    \\
   \grad(\fpc\unkd)
 & \wto\grad(\fpc\unk)
    \,,
    \qquad
    & \textup{weakly in $L^{2}(\OoutT)$;}
    \\
    \pder{\unkd}{t}
    &
    \wto\pder{\unk}{t}
    \,,
    \qquad
    & \textup{weakly in $L^{2}(\Oout\times(t_{0},T))$,}
  \end{aligned}
  \end{equation}
  for a suitable $\unk\in L^{2}(\OoutT)$.
  In addition
  \begin{equation}
    \label{eq:unkd_int_weak}
    \unkd
    \wto
    \unk
    \,,
    \qquad
    \textup{weakly in $L^{2}(0,T;L^{2}_{\textupmd{loc}}(\Oint))$,}
  \end{equation}
  for a suitable $\unk\in L^{2}_{\textupmd{loc}}(\Oint\times[0,T])$.

  The limits of $\unkd$ solve the problems \eqref{eq:pde_2}--\eqref{eq:init_2} and \eqref{eq:pde_1}--\eqref{eq:init_1} respectively.
\end{theorem}

\begin{proof}
Let us recall the notation $\funkd=\fpcd\unkd$, 
i.e., $\funkd=\fpc\unkd$ in $\Oout_T$, 
and setting $\unk=\funk/\fpc$ in $\Oout_T$, 
as a consequence of Remark~\ref{r:vcomp_ext} we have \eqref{eq:unkd_out_strong}.

Next we show that $\unk$ solves in the weak sense 
\eqref{eq:pde_2}--\eqref{eq:init_2}.
First note that
the function $\funkd/\sqrt{\fpcd}$ converges weakly in $L^{2}(\OsetT)$, owing to \eqref{eq:energy_v}. Therefore, $\funk=0$ in $\OintT$, i.e., $\funkd$, $\grad\funkd\wto 0$ weakly in $L^{2}(\OintT)$.

  Thus, by continuity of traces, as $\delta\to0$,
  \begin{equation}
    \label{eq:jump_lim}
    0
    =
    [\fpcd\unkd]
    =
    \funkd^{\outtr}\restr{\inter}
    -
    \funkd^{\inttr}\restr{\inter}
    \to
    \funk^{\outtr}\restr{\inter}
    =
    (\fpc\unk)^{\outtr}\restr{\inter}
    \,.
  \end{equation}
  This proves \eqref{eq:jump}.

Moreover, 
in \eqref{eq:weaksol_d_a} take $\varphi\in H^{1}(\OoutT)$, 
with $\varphi=0$ on $\inter$ and at $t=T$ and extend to $0$ in 
$\Oint_T$. We obtain
  \begin{equation}
    \label{eq:weaksol_d_out}
    \int_{\OoutT}
    \Big[
    -
    \unkd
    \pder{\varphi}{t}
    +
    \grad(\fpc\unkd)
    \grad \varphi
    \Big]
    \di x
    \di t
    =
    \int_{\Oout}
    \initd
    \varphi(0)
    \di x
    +
    \int_{\OoutT}
    f\varphi
    \di x
    \di t
    \,.
  \end{equation}
  As $\delta\to 0$ we get \eqref{eq:weaksol_out}.

  As to the problem in the interior domain $\Oint$, we remark that from \eqref{eq:energy_dint_n} our claim \eqref{eq:unkd_int_weak} follows.

Next, we prove that $\unk$ solve \eqref{eq:pde_1}--\eqref{eq:init_1} 
weakly.
Consider $\varphi\in \mathcal{C}^{1}(\Oint_T)$ 
and such that its support is bounded away from $\partial\Oint$ and $t=T$.
From \eqref{eq:weaksol_d_a} we have 
  \begin{multline}
    \label{eq:unkdt_int_weak}
    \int_{\OintT}
    f
    \varphi
    \di x
    \di t
    +
    \int_{\Oint}
    \initd
    \varphi(0)
    \di x
    =
    -
    \int_{\OintT}
    \unkd
    \pder{\varphi}{t}
    \di x
    \di t
    +
    \int_{\OintT}
    \grad\funkd
    \grad\varphi
    \di x
    \di t
    \\
    \to
    -
    \int_{\OintT}
    \unk
    \pder{\varphi}{t}
    \di x
    \di t
    \,.
  \end{multline}
From \eqref{eq:unkdt_int_weak}, 
standard arguments prove that $\unk$ is given by $F$ where $F$ 
is defined in \eqref{eq:add_F}.
\end{proof}


\subsection{Limiting behavior in the whole domain.}
\label{s:add}
We point out that, as we will show below,  
$L^{2}$ convergence can not take place in our case
in the whole domain $\Oset_T$.
We investigate here the concentration of mass on $\inter$ as $\delta\to 0$.
Here we denote for $1>\sigma>0$
\begin{equation*}
  \Oint(\sigma)
  =
  \{
  x\in\Oint
  \mid
  \dist(x,\inter)<\sigma
  \}
  \,,
\end{equation*}
so that
\begin{equation*}
  \abs{\Oint(\sigma)}
  \le
  \gamma
  \sigma
  \,.
\end{equation*}

The next Lemma is independent of the convergence results 
of Theorem~\ref{t:unkd_limit} and relies on the degenerating diffusion in $\Oint$ as $\delta\to0$.

\begin{lemma}
  \label{l:add_expand}
  We have for all fixed $0<\sigma<1$ and $0<p<1$,
  \begin{equation}
    \label{eq:add_expand_n}
    \int_{\Oint\setminus\Oint(\sigma)}
    (\unkd(t)-F(t))^{2}
    \di x
    \le
    \gamma
    e^{\gamma t}
    [
    \delta
    (\delta^{-p}+\sigma^{-4})
    +
    c(\delta)
    ]
    \,,
    \qquad
    t>0
    \,,
  \end{equation}
where $F$ has been defined in \eqref{eq:add_F}.
Here $\gamma$ is a constant depending on 
$T$, $\|\bar u\|_{L^2(\Oset)}$, $\|f\|_{L^2(\Oset)}$, 
but not on $\sigma$, $p$, and $\delta$. 
Moreover, 
$c(\delta)\to0$ as $\delta\to0$.
\end{lemma}

\begin{proof}
  We introduce smooth approximations $(\fpc\initd)_{\delta}$, $(\fpc f)_{\delta}$ such
  that $(\fpc\initd)_{\delta}\to \fpc\initd$ in $L^{2}(\Oset)$, $(\fpc f)_{\delta}\to \fpc f$ in $L^{2}(\OsetT)$ as
  $\delta\to0$. Then we set $\initdd=(\fpc\initd)_{\delta}/\fpc\to \initd$, $f_{\delta}=(\fpc f)_{\delta}/\fpc \to f$ in
  $L^{2}(\Oset)$; we may assume without loss of generality that
  \begin{equation}
    \label{eq:add_expand_jj}
    \norma{\grad (\fpc\initdd)}{L^{2}(\Oset)}^{2}
    +
    \norma{\grad (\fpc f_{\delta})}{L^{2}(\Oset)}^{2}
    \le
    \gamma
    \delta^{-p}
    \,,
  \end{equation}
  for $p$ as above, by relabeling if necessary 
  the sequences $\initdd$, $f_{\delta}$.
  Define
  \begin{equation*}
    F_{\delta}(x,t)
    =
    \initdd(x)
    +
    \int_{0}^{t}
    f_{\delta}(x,\tau)
    \di \tau
    \,,
    \qquad
    (x,t)\in \OsetT
    \,.
  \end{equation*}
  Use in the weak formulation \eqref{eq:weaksol_d_a}
  the test function $\fpc(\unkd - F_{\delta})\varphi^{2}$, with
  $\varphi\in \mathcal{C}^{1}_{0}(\Oint)$, and
  \begin{equation*}
    \varphi(x)
    =
    1
    \,,
    \quad
    x\not\in \Oint(\sigma)
    \,;
    \quad
    \varphi(x)
    =
    0
    \,,
    \quad
    x\in \Oint(\sigma/2)
    \,;
    \quad
    \abs{\grad\varphi}
    \le
    \gamma
    \sigma^{-1}
    \,.
  \end{equation*}
Note that this test function has the required regularity 
due to the definitions above and to the fact that $\fpc$ is independent
of $t$.
  After routine calculations starting from \eqref{eq:weaksol_d_a}, we find
  \begin{equation*}
      \frac{1}{2}
      \int_{\Oint}
      (\unkd(t)-F_{\delta}(t))^{2}
      \varphi^{2}
      \fpc
      \di x
      +
      \delta
      \int_{0}^{t}
      \int_{\Oint}
      \abs{\grad[\fpc(\unkd-F_{\delta})]}^{2}
      \varphi^{2}
      \di x
      \di \tau
      =
      \sum_{h=1}^{3}
      E_{h}
      \,.
  \end{equation*}
  Here, the term with integration in time contributes
  \begin{multline*}
    E_{1}
    =
    \int_{0}^{t}
    \int_{\Oint}
    (f-f_{\delta})
    (\unkd-F_{\delta})
    \varphi^{2}
    \fpc
    \di x
    \di \tau
    +
    \frac{1}{2}
    \int_{\Oint}
    (\initd-\initdd)^{2}
    \varphi^{2}
    \fpc
    \di x
    \\
    \le
    \gamma
    \norma{f-f_{\delta}}{L^{2}(\OsetT)}^{2}
    +
    \gamma
    \norma{\initd-\initdd}{L^{2}(\Oset)}^{2}
    +
    \int_{0}^{t}
    \int_{\Oint}
    (\unkd-F_{\delta})^{2}
    \varphi^{2}
    \di x
    \di \tau
    \,.
  \end{multline*}
  Moreover, we added the following term to construct the correct energy:
  \begin{multline*}
    E_{2}
    =
    -
    \delta
    \int_{0}^{t}
    \int_{\Oint}
    \varphi^{2}
    \grad(\fpc F_{\delta})
    \grad[\fpc(\unkd-F_{\delta})]
    \di x
    \di \tau
    \\
    \le
    \frac{\delta}{2}
    \int_{0}^{t}
    \int_{\Oint}
    \abs{\grad[\fpc(\unkd-F_{\delta})]}^{2}
    \varphi^{2}
    \di x
    \di \tau
    +
    \frac{\delta}{2}
    \int_{0}^{t}
    \int_{\Oint}
    \abs{\grad(\fpc F_{\delta})}^{2}
    \varphi^{2}
    \di x
    \di \tau
    \,.
  \end{multline*}
Finally, from the integration by parts we get the term 
  \begin{multline*}
    E_{3}
    =
    -
    2
    \delta
    \int_{0}^{t}
    \int_{\Oint}
    (\unkd- F_{\delta})
    \varphi
    \fpc
    \grad(\fpc\unkd)
    \grad\varphi
    \di x
    \di \tau
    \\
    \le
    \gamma
    \delta
    \norma{\grad\varphi}{\infty}^{2}
    \int_{0}^{t}
    \int_{\supp\varphi}
    \delta
    \abs{\grad(\fpc\unkd)}^{2}
    \di x
    \di \tau
    +
    \int_{0}^{t}
    \int_{\Oint}
    (\unkd-F_{\delta})^{2}
    \varphi^{2}
    \di x
    \di \tau
    \\
    \le
    \gamma
    \delta
    \sigma^{-4}
    +
    \int_{0}^{t}
    \int_{\Oint}
    (\unkd-F_{\delta})^{2}
    \varphi^{2}
    \di x
    \di \tau
    \,,
  \end{multline*}
where the bound is a consequence of \eqref{eq:energy_dint_n} 
for a suitable cut off function identically equal to $1$ over $\supp\varphi$ 
and with gradient still bounded by $\gamma/\sigma$.

  Thus collecting all the estimates above we get
  \begin{multline*}
    \int_{\Oint}
    (\unkd(t)-F_{\delta}(t))^{2}
    \varphi^{2}
    \di x
    \le
    \gamma
    \norma{f-f_{\delta}}{L^{2}(\OsetT)}^{2}
    +
    \gamma
    \norma{\initd-\initdd}{L^{2}(\Oset)}^{2}
    \\
    +
    \gamma
    \delta(\delta^{-p}+\sigma^{-4})
    +
    \gamma
    \int_{0}^{t}
    \int_{\Oint}
    (\unkd-F_{\delta})^{2}
    \varphi^{2}
    \di x
    \di \tau
    \,.
  \end{multline*}
  The claim follows now from Gronwall's lemma and from the obvious fact
  \begin{multline*}
    \int_{\Oint}
    (\unkd(t)-F(t))^{2}
    \varphi^{2}
    \di x
    \le
    2
    \int_{\Oint}
    [
    (\unkd(t)-F_{\delta}(t))^{2}
    +
    (F_{\delta}(t)-F(t))^{2}
    ]
    \varphi^{2}
    \di x
    \\
    \le
    2
    \int_{\Oint}
    (\unkd(t)-F_{\delta}(t))^{2}
    \varphi^{2}
    \di x
    +
    4
    \norma{\initd-\initdd}{L^{2}(\Oset)}^{2}
    +
    4T
    \norma{f-f_{\delta}}{L^{2}(\OsetT)}^{2}
    \,.
  \end{multline*}
\end{proof}

\begin{corollary}
  \label{co:add_expand}
  We have uniformly for $0\le t\le T$,
  \begin{equation}
    \label{eq:add_expand_m}
    \lim_{\delta\to 0}
    \int_{\Oint\setminus\Oint(\delta^{q})}
    (\unkd(t)-F(t))^{2}
    \di x
    =
    0
    \,,
  \end{equation}
  for any fixed $0<q<1/4$.
\end{corollary}

\begin{proof}
  We need only take $\sigma=\delta^{q}$ in \eqref{eq:add_expand_n}.
\end{proof}

Estimates \eqref{eq:energy_d_n}, \eqref{eq:energyt0_d_n} 
yield similar inequalities 
for $\unk$ in $\Oout$, and therefore on invoking, e.g., 
the result 
\cite[Theorem~8.12]{Gilbarg:Trudinger:1983}, 
we have that $\fpc\unk \in L^{2}(t_{0},T;H^{2}(\Oout))$ for every $t_{0}>0$, so that $\grad (\fpc\unk)$ has a trace in $L^{2}(t_{0},T;H^{1/2}(\inter))$. 
However, the function
\begin{equation}
  \label{eq:elliptic_w}
  w(x;t)
  =
  \int_{0}^{t}
  \fpc(x)
  \unk(x,\tau)
  \di\tau
  \,,
  \qquad
  x\in\Oout
  \,,
\end{equation}
solves for each fixed $t>0$ the elliptic equation
\begin{equation}
  \label{eq:elliptic_pde}
  \Lapl w
  =
  -
  \int_{0}^{t}
  f(\tau)
  \di\tau
  -
  \initd
  +
  \unk(t)
  \,,
  \qquad
  \text{in $\Oout$,}
\end{equation}
with null Dirichlet data on $\inter$ and null Neumann data on
$\bdr{\Oset}$. This is a consequence of \eqref{eq:weaksol_out} and of
a suitable choice of factorized test functions.
Then $w(t)\in H^{2}(\Oout)$, and
\begin{equation}
  \label{eq:elliptic_flux}
  \int_{\Oset}
  \varphi
  \di\msr{bu}{t}
  :=
  \int_{\inter}
  \pder{w}{\normint}
  \varphi
  \di S
  \qquad
  \varphi\in \mathcal{C}(\overline{\Oset})
  \,,
\end{equation}
is defined as a linear bounded functional e.g., on $\mathcal{C}(\overline{\Oset})$, that is a measure.   
The flux of $\fpc \unk$ through $\inter$ should be in general understood in this sense, though
of course the normal derivative $\partial (\fpc\unk)/\partial n$ on $\inter$ exists in the classical sense under suitable regularity assumptions on the data.

\begin{theorem}
  \label{t:add_limit}
  We have for $0<q<1/4$ and for all $\varphi\in \mathcal{C}^{1}(\overline{\Oset})$
  \begin{equation}
    \label{eq:add_limit_n}
    \lim_{\delta\to0}
    \int_{\Oint(\delta^{q})}
    \unkd(t)
    \varphi
    \di x
    =
    \int_{\Oset}
    \varphi
    \di\msr{\fpc\unk}{t}
    \,,
  \end{equation}
  where $\unk$ is the solution to \eqref{eq:pde_2}--\eqref{eq:init_2}.
\end{theorem}

\begin{proof}
  Use as a testing function in \eqref{eq:weaksol_d_a} $\varphi\in \mathcal{C}^{1}(\overline{\Oset})$, after an approximation process as in the proof of Lemma~\ref{l:energy_dint}. We get for every $t>0$
  \begin{equation}
    \label{eq:add_limit_i}
    J_{1}
    +
    J_{2}
    =
    \int_{\Oset}
    F(t)
    \varphi
    \di x
    \,,
  \end{equation}
  where
  \begin{align*}
    J_{1}
    =
    \int_{\Oint}
    \unkd (t)
    \varphi
    \di x
    +
    \int_{0}^{t}
    \int_{\Oint}
    \delta
    \grad(\fpc\unkd)
    \grad\varphi
    \di x
    \di \tau
    =:
    J_{11}
    +
    J_{12}
    \,,
    \\
    J_{2}
    =
    \int_{\Oout}
    \unkd (t)
    \varphi
    \di x
    +
    \int_{0}^{t}
    \int_{\Oout}
    \grad(\fpc\unkd)
    \grad\varphi
    \di x
    \di \tau
    \,.
  \end{align*}
  Clearly
  \begin{equation}
    \label{eq:add_limit_ii}
    J_{2}
    \to
    \int_{\Oout}
    \unk (t)
    \varphi
    \di x
    +
    \int_{0}^{t}
    \int_{\Oout}
    \grad(\fpc\unk)
    \grad\varphi
    \di x
    \di \tau
    \,,
    \qquad
    \delta\to 0
    \,,
  \end{equation}
  for $\unk$ as in Theorem~\ref{t:unkd_limit} (on using also the last of \eqref{eq:unkd_out_strong}).

  As to $J_{1}$, clearly $J_{12}\to 0$ as $\delta\to 0$ by virtue of the weak convergence $\grad \funkd\wto 0$ in $L^{2}(\OintT)$ given in the proof of Theorem~\ref{t:unkd_limit}. In addition
  \begin{equation}
    \label{eq:add_expand_iii}
    J_{11}
    =
    \int_{\Oint(\delta^{q})}
    \unkd(t)
    \varphi
    \di x
    +
    \int_{\Oint\setminus\Oint(\delta^{q})}
    \unkd(t)
    \varphi
    \di x
    \,,
  \end{equation}
  and the last integral above converges according to \eqref{eq:add_expand_m}.

  Thus as $\delta\to 0$
  \begin{multline}
    \label{eq:add_limit_iv}
    \int_{\Oint(\delta^{q})}
    \unkd(t)
    \varphi
    \di x
    \to
    \int_{\Oset}
    F(t)
    \varphi
    \di x
    -
    \int_{\Oout}
    \unk (t)
    \varphi
    \di x
    \\
    -
    \int_{0}^{t}
    \int_{\Oout}
    \grad(\fpc\unk)
    \grad\varphi
    \di x
    \di \tau
    -
    \int_{\Oint}
    F(t)
    \varphi
    \di x
    \\
    =
    \int_{\Oout}
    F(t)
    \varphi
    \di x
    -
    \int_{\Oout}
    \unk (t)
    \varphi
    \di x
    -
    \int_{0}^{t}
    \int_{\Oout}
    \grad(\fpc\unk)
    \grad\varphi
    \di x
    \di \tau
    =:
    J_{3}
    \,.
  \end{multline}
  On the other hand, we take into account the weak formulation of \eqref{eq:elliptic_pde}, for $w$ as in \eqref{eq:elliptic_w}, where we may allow test functions $\varphi\in \mathcal{C}^{1}(\overline{\Oout})$ (not necessarily vanishing on $\inter$), since the regularity of $w$ implies $\grad w(t) \in L^{2}(\inter)$. Then, recalling \eqref{eq:elliptic_flux} we arrive at
  \begin{equation}
    \label{eq:add_limit_v}
    J_{3}
    =
    \int_{\Oset}
    \varphi
    \di\msr{\fpc\unk}{t}
    \,,
  \end{equation}
  whence the claim.
\end{proof}

From the results of this Section it follows 
immediately the following corollary.

\begin{corollary}
  \label{c:add_total}
  As $\delta \to 0$ the solution $\unkd$ satisfies for every $0<t<T$ and every $\varphi\in \mathcal{C}(\overline{\Oset})$
  \begin{equation}
    \label{eq:total_n}
    \int_{\Oset}
    \unkd(x,t)
    \varphi(x)
    \di x
    \to
    \int_{\Oset}
    \varphi(x)
    \di \finmass_{t}(x)
    \,,
  \end{equation}
  where
  \begin{equation}
    \label{eq:total_nn}
    \di\finmass_{t}
    :=
    \unk(x,t)
    \di x
    +
    \di \msr{bu}{t}
    \,,
  \end{equation}
  and $\unk$ is the limiting solution, defined in the whole $\OsetT$, introduced in Theorem~\ref{t:unkd_limit}.
  \end{corollary}
    \begin{proof}
      In fact this follows at once from Theorem~\ref{t:add_limit}, when we exploit the uniform $L^{1}$ bound of Lemma~\ref{l:mass_con} and a standard approximation procedure of $\varphi$ with $\mathcal{C}^{1}$ functions.
    \end{proof}
  
\section{Unfolding}
\label{s:unfold}
In this subsection we recall the definition and the 
main properties concerning the usual periodic-unfolding 
operator and the
unfolding operator for perforated domains, see, for instance, 
\cite{Amar:Andreucci:Gianni:Timofte:2017A,Amar:Gianni:2016A,
Cabarrubias:Donato:2016,
Cioranescu:Damlamian:Griso:2002,Cioranescu:Damlamian:Griso:2008,Cioranescu:Damlamian:Griso:Onofrei:2008}.

\begin{definition}
\label{d:fold010}
The time--depending unfolding operator $\mathcal{T}_\varepsilon$
of a 
function $w$ defined on $\Omega_T$
and 
Lebesgue measurable 
is given by
\begin{equation}
\label{fold020b}
\mathcal{T}_\varepsilon(w)(x,t,y)
=
\left\{
\begin{alignedat}{2}
&
w\Big(\varepsilon\Big[\frac{x}{\varepsilon}\Big]_\Y
     +\varepsilon y,t \Big)
,
&
\quad
&
(x,t,y)\in\hat\Omega_{\varepsilon,T}\times\Y;
\\
&
0,
&
\quad
&
\text{otherwise}
.
\end{alignedat}
\right.
\end{equation}
\end{definition}

Note that, by definition, it easily follows that
\begin{equation}
\label{fold080}
\mathcal{T}_\varepsilon(w_1w_2)
=
\mathcal{T}_\varepsilon(w_1)
\mathcal{T}_\varepsilon(w_2)\,.
\end{equation}

\begin{definition}
\label{d:fold000}
The space average operator $\mathcal{M}_\varepsilon$
of a Lebesgue integrable function $w$ defined on $\Omega_T$
is given by
\begin{equation}
\label{fold020}
\mathcal{M}_\varepsilon(w)(x,t)
=
\left\{
\begin{alignedat}{2}
&
\frac{1}{\eps^{N}}\int_{\Y_\eps(x)}w(\zeta,t)\di \zeta
,
&
\quad
&
(x,t)\in\hat\Omega_{\varepsilon,T};
\\
&
0,
&
\quad
&
\text{otherwise}
.
\end{alignedat}
\right.
\end{equation}
Moreover, the space oscillation operator
is defined as
\begin{equation}
\label{fold070}
\mathcal{Z}_\varepsilon(w)(x,t,y)
=
\mathcal{T}_\varepsilon(w)(x,t,y)
-
\mathcal{M}_\varepsilon(w)(x,t)
\;.
\end{equation}
\end{definition}

Notice that, by a simple change of variables, it easily follows that
\begin{equation}
\label{fold050}
\mathcal{M}_\varepsilon(w)(x,t)
=
\int_\Y
  \mathcal{T}_\varepsilon(w)(x,t,y)
\,\text{d}y\,
\,.
\end{equation}

For later use, we define the functional spaces
\begin{equation}\label{d:space}
\Hper(\Y)=\{\vi\in H^1_\textup{loc}(\mathbb{R}^n):\ \hbox{$\vi$ is $\Y$-periodic}\}\,,
\end{equation}
and
\begin{equation}\label{d:space1}
K_\hole=\{\vi\in L^{2^*}(\mathbb{R}^n):\ \nabla \vi\in L^2(\R^n),\ \hbox{$\vi$ is constant on $\hole$}\}\,.
\end{equation}

We recall the following results.

\begin{proposition}
\label{fold090}
For $\phi\in L^2(\Y;{\mathcal C}(\overline\Omega_T))$
or
$\phi\in L^2(\Omega_T;{\mathcal C}(\overline{\Y}))$,
denote again by $\phi$ its
extension by $\Y$--periodicity
to $\Omega_T\times\mathbb{R}^{n}$
and set
$\phi_\varepsilon(x,t)=\phi(x,t,\varepsilon^{-1}x)$.
Then,
$\mathcal{T}_\varepsilon(\phi_\varepsilon) \to \phi$
strongly in $L^2(\Omega_T\times\Y)$.
\end{proposition}

\begin{proposition}
\label{p:fold2}
Let $w_\eps\to w$ strongly in $L^2(\Om_T)$. Then
\begin{equation}\label{eq:m1bis}
\mathcal{T}_\varepsilon(w_\eps)\to w\,,\qquad  \hbox{strongly in $L^2(\Om_T\times\Y)$.}
\end{equation}
\end{proposition}

\begin{proposition}
\label{p:fold1}
Let $w_\eps\wto w$ weakly in $L^2(0,T;H^1(\Om))$. Then there exists a function $\hat w\in L^2(\Om_T;\Hper(\Y))$
such that
\begin{equation}\label{eq:m1}
\begin{aligned}
& \mathcal{T}_\varepsilon(w_\eps)\wto w\,,\qquad & \hbox{weakly in $L^2(\Om_T;H^1(\Y))$;}
\\
& \mathcal{M}_\varepsilon(w_\eps)\wto w\,,\qquad & \hbox{weakly in $L^2(\Om_T)$;}
\\
&
\mathcal{T}_\varepsilon(\nabla w_\eps)\wto \nabla w+\nabla_y \hat w\,,\qquad & \hbox{weakly in $L^2(\Om_T\times\Y)$.}
\end{aligned}
\end{equation}
\end{proposition}
Due to the presence of small holes, 
we are led to introduce another unfolding operator, depending also on 
the size of the small holes.
It is denoted by $\unfoldee$ and defined as
\begin{equation}
\label{eq:m10}
\begin{aligned}
\unfoldee(w)(x,t,z)
& =
{\left\{
\begin{alignedat}{2}
&
\unfolde(w)(x,t,\eta z)
,
&
\quad
&
(x,t,z)\in\hat\Omega_{\varepsilon,T}\times\frac{1}{\eta}\Y;
\\
&
0,
&
\quad
&
\text{otherwise};
\end{alignedat}
\right.}
\\
&\\
& =
{\left\{\begin{alignedat}{2}
& w\Big(\varepsilon\Big[\frac{x}{\varepsilon}\Big]_\Y
     +\varepsilon \eta z,t \Big),
&
\quad
&
(x,t,z)\in\hat\Omega_{\varepsilon,T}\times\frac{1}{\eta}\Y;
\\
&
0,
&
\quad
&
\text{otherwise}.
\end{alignedat}
\right.}
\end{aligned}
\end{equation}
The operator $\unfoldee$ satisfies property \eqref{fold080}, too; moreover, for every $w\in L^1(\Om_T)$,
we have
\begin{alignat}2
\label{eq:m11}
& \int_{\Om_T\times\R^n} |\unfoldee(w)|\di x   \di t\di z\leq \frac{1}{\eta^n}\int_{\Om_T} |w|\di x\di t,
\\
\label{eq:m12}
& \left| \int_{\Om_T} w\di x\di t-\eta^n\int_{\Om_T\times\R^n}\unfoldee(w)\di x\di t\di z\right|\leq\int_{\Lambda_{\eps,T}}|w|\di x\di t.
\end{alignat}
For $w\in L^2(0,T;H^1(\Om))$, we have
\begin{alignat}2
\label{eq:m13}
& \unfoldee (\nabla w) = \frac{1}{\eps\eta}\nabla_z(\unfoldee(w)),\quad \hbox{in $\Om_T\times \frac{1}{\eta}\Y$;}
\\
\label{eq:m14}
& \Vert \nabla_z(\unfoldee(w))\Vert^2_{L^2(\Om_T\times\frac{1}{\eta}\Y)}\leq \frac{\eps^2}{\eta^{n-2}}\Vert \nabla w\Vert^2_{L^2(\Om_T)};
\\
\label{eq:m15}
& \Vert \unfoldee(w-\mathcal{M}_\eps(w))\Vert^2_{L^2(\Om_T;L^{2^*}(\R^n)))}\leq C\frac{\eps^2}{\eta^{n-2}}\Vert \nabla w\Vert^2_{L^2(\Om_T)};
\\
\label{eq:m16}
& \Vert \unfoldee(w)\Vert^2_{L^2(\Om_T\times\omega)}\leq \frac{2C\eps^2}{\eta^{n-2}}|\omega|^{2/n}\Vert \nabla w\Vert^2_{L^2(\Om_T)}
+2|\omega|\,\Vert w\Vert^2_{L^2(\Om_T)}.
\end{alignat}
Here,
$\omega\subset\R^n$ is a bounded open set, $C$ is the Sobolev-Poincar\'{e}-Wirtinger constant for $H^1(\Y)$,
and properties \eqref{eq:m14}--\eqref{eq:m16} hold for $n\geq 3$.
\medskip

Finally, we recall the following compactness result.

\begin{proposition}\label{p:p1}
Let $n\geq 3$ and $\{w_{\eps}\}\subset L^2(0,T;H^1(\Om))$ be a uniformly  bounded sequence.
Then, up to a subsequence, there exists $W\in L^2(\Om_T;L^{2^*}(\R^n))$, with $\nabla_z W\in L^2(\Om_T\times\R^n)$, such that
\begin{alignat}2
\label{eq:m20}
& \!\!\!\!\!\!\frac{\eta^{n/2-1}}{\eps}(\unfoldee(w_{\eps})\!-\!\mathcal{M}_\eps(w_{\eps})\chi_{\frac{1}{\eta}\Y})\!\wto\! W\
&\hbox{weakly in $L^2(\Om_T;L^{2^*}\!\!(\R^n))$;}
\\
\label{eq:m17}
& \!\!\!\!\!\!\frac{\eta^{n/2-1}}{\eps}\nabla_z(\unfoldee(w_{\eps}))\chi_{\frac{1}{\eta}\Y}\wto \nabla_z W\
&\hbox{weakly in $L^2(\Om_T\times\R^n)$.}
\end{alignat}
Moreover, if
$$
\limsup_{\eps\to 0}\frac{\eta^{n/2-1}}{\eps}<+\infty,
$$
one can choose the subsequence above and some $V\in L^2(\Om_T;L^2_{loc}(\R^n))$ such that
\begin{equation}\label{eq:m18}
\frac{\eta^{n/2-1}}{\eps}\unfoldee(w_{\eps})\wto V\quad
\hbox{weakly in $L^2(\Om_T;L^2_{loc}(\R^n))$.}
\end{equation}
\end{proposition}

\section{Homogenization of the degenerate problem}
\label{ss:homog_deg}

In this section we assume $n\ge3$.
According to the results of Section~\ref{s:fp_nondeg},  we are 
interested here in homogenizing the following problem; in the 
outer domain we have
\begin{alignat}{2}
  \label{eq:pde_2_eps}
  \pder{\unk_\eps}{t}
  -
  \Lapl(\fpc_\eps\unk_\eps)
  &=
  f
  \,,
  &\qquad&
  \text{in $\Oout_{\eps,T}$;}
  \\
  \label{eq:bdr_2_eps}
  \pder{(\fpc_\eps\unk_\eps)}{\normal}
  &=
  0
  \,,
  &\qquad&
  \text{on $\bdr{\Oset}\times(0,T)$;}
  \\
  \label{eq:jump_eps}
  \fpc_\eps\unk_\eps
  &=
  0
  \,,
  &\qquad&
  \text{on $\MembT$;}
  \\
  \label{eq:init_2_eps}
  \unk_\eps(x,0)
  &=
  \initd(x)
  \,,
  &\qquad&
  \text{in $\Oout_\eps$.}
\end{alignat}
The problem in the interior domain is
\begin{alignat}{2}
  \label{eq:pde_1_eps}
  \pder{\unk_\eps}{t}
  &=
  f
  \,,
  &\qquad&
  \text{in $\Oint_{\eps,T}$;}
  \\
  \label{eq:init_1_eps}
  \unk_\eps(x,0)
  &=
  \initd(x)
  \,,
  &\qquad&
  \text{in $\Oint_\eps$.}
\end{alignat}
Here, $\fpc_\eps(x)=\fpc(x,x/\eps)$ in $\Om$,
where $\fpc$ is the Carath\'{e}odory function introduced in Section \ref{s:problem}.
Clearly, \eqref{eq:pde_1_eps}, \eqref{eq:init_1_eps} lead to
\begin{equation}
  \label{eq:F_1_eps}
  \unk_\eps(x,t)
  =
  F(x,t)
  \,,
  \qquad
  (x,t)\in \Oint_{\eps,T}
  \,,
\end{equation}
for the $F$ defined in \eqref{eq:add_F}.

If we set $v_\eps=\fpc_\eps u_\eps$ in $\Oout_{\eps,T}$, we can rewrite  problem \eqref{eq:pde_2_eps}--\eqref{eq:init_2_eps} as
\begin{alignat}{2}
  \label{eq:m5_eps}
  \frac{1}{\fpc_\eps}\pder{v_\eps}{t}
  -
  \Lapl v_\eps
  &=
  f
  \,,
  &\qquad&
  \text{in $\Oout_{\eps,T}$;}
  \\
  \label{eq:m6_eps}
  \pder{v_\eps}{\normal}
  &=
  0
  \,,
  &\qquad&
  \text{on $\bdr{\Oset}\times(0,T)$;}
  \\
  \label{eq:m7_eps}
  v_\eps
  &=
  0
  \,,
  &\qquad&
  \text{on $\MembT$;}
  \\
  \label{eq:m8_eps}
  v_\eps(x,0)
  &=
  \overline v_{\eps}(x)
  \,,
  &\qquad&
  \text{in $\Oout_\eps$,}
\end{alignat}
where $\overline v_{\eps}=\fpc_\eps\initd\in L^2(\Om)$.
Thanks to \eqref{eq:m7_eps},
$v_\eps$ can be naturally extended to zero in $\Oint_\eps$ and this extension belongs to $L^2(0,T;H^1(\Om))$.
Consequently, from now on, we will not distinguish between $v_\eps$ defined on $\Oout_{\eps,T}$ and its
extension on the whole of $\Om_T$. The same identification will be adopted for all the functions
which are null on $\Memb$ or $\MembT$; note that this does not apply to $u_{\eps}$.
Moreover, the weak formulation of problem
\eqref{eq:m5_eps}--\eqref{eq:m8_eps} is given by
\begin{multline}\label{eq:weak_form1}
-\int_{\Oout_{\eps,T}}\frac{v_\eps}{\fpc_\eps}\phi_t\di x\di t+\int_{\Oout_{\eps,T}} \nabla v_\eps\nabla\phi\di x\di t
\\
=\int_{\Oout_{\eps,T}} f\phi\di x\di t +\int_{\Oout_\eps} \frac{\overline v_{\eps}}{\fpc_\eps}\phi(x,0)\di x\,,
\end{multline}
for all test functions $\phi\in H^1(\Oout_{\eps,T})$, with $\phi=0$ on $\MembT$ and $\phi(x,T)=0$ in $\Om$ in the sense of traces.
By taking $\phi=v_\eps$ in \eqref{eq:weak_form1} and 
recalling that $v_\eps$
is null in $\Oint_{\eps,T}$, 
by
using standard approximations, we get
\begin{equation}\label{eq:energy1}
\frac{1}{2}\int_{\Om}\frac{v^2_\eps(x,t)}{\fpc_\eps(x)}\di x+\int_{\Om_T} |\nabla v_\eps|^2\di x\di t
=\int_{\Om_T} f v_\eps\di x\di t +\int_{\Oout_\eps}\fpc_{\eps}{\initd}^2\di x\,.
\end{equation}
Recalling
\eqref{eq:m7}, by using Young and Gronwall inequalities,
we arrive at the standard energy inequality
\begin{equation}\label{eq:energy2}
\sup_{t\in(0,T)}\int_{\Om}v^2_\eps\di x+\int_{\Om_T} |\nabla v_\eps|^2\di x\di t
\leq\const (\Vert f\Vert^2_{L^2(\Om_T)}+\Vert \initd\Vert^2_{L^2(\Om)}),
\end{equation}
where $\const$ is independent of $\eps$.
As a consequence of \eqref{eq:energy2} and of \eqref{eq:F_1_eps} we get also
\begin{equation}\label{eq:energy3}
  \sup_{t\in(0,T)}\int_{\Om}u^2_\eps\di x\leq\const (\Vert f\Vert^2_{L^2(\Om_T)}+\Vert \initd\Vert^2_{L^2(\Om)})
  \,.
\end{equation}
This implies that there exist $u_0\in L^2(\Om_T)$ and $v_0\in L^2(0,T;H^1(\Om))$ such that, up to
a subsequence,
\begin{equation}\label{eq:m8}
\begin{aligned}
& u_\eps\wto u_0\qquad &\hbox{weakly in $L^2(\Om_T)$;}
\\
& v_\eps\wto v_0\qquad &\hbox{weakly in $L^2(\Om_T)$;}
\\
& \nabla v_\eps\wto \nabla v_0\qquad &\hbox{weakly in $L^2(\Om_T)$.}
\end{aligned}
\end{equation}

  \begin{remark}
    \label{r:v_strong}
  Indeed, exactly as in \eqref{eq:energy_v}, one can easily prove versions of \eqref{eq:energy1} and \eqref{eq:energy2} yielding also an estimate of $\partial_{t}v_{\eps}$ in $L^{2}_{loc}(t_{0},T;\Oset)$ which is uniform in $\eps$ (though not in $t_{0}$). Thus we may claim here essentially the same convergence as in Remark~\ref{r:vcomp_ext}, namely $v_{\eps}\to v_{0}$ strongly in $L^{2}(\OsetT)$ as $\eps\to0$.
\end{remark}

\begin{remark}
  \label{r:deg_mass}
  Owing to Corollary~\ref{c:add_total}, we have that the total mass in the degenerate problem \eqref{eq:pde_2_eps}--\eqref{eq:init_1_eps} is in fact represented by the measure
  \begin{equation}
    \label{eq:deg_mass_a}
    u_{\eps}(x,t)
    \di x
    +
    \di\msr{\fpc_{\eps}u_{\eps}}{t}
    \,,
    \qquad
    \text{in $\Oset$,}
  \end{equation}
  where $\msr{\fpc_{\eps}u_{\eps}}{t}$ is defined as in \eqref{eq:elliptic_flux} (here of course $\inter$ is replaced with $\Memb$). 
\end{remark}

\subsection{The limit equation in the critical case $\eta\approx\eps^{2/(n-2)}$}
\label{s:crit}
We assume that $\eta=\eta(\eps)$ satisfies
\begin{equation}\label{eq:m32}
\lim_{\eps\to 0}\frac{\eta^{n/2-1}}{\eps}=k\in (0,+\infty).
\end{equation}

\begin{theorem}\label{t:t1}
Let $\{v_\eps\}\subset L^2(0,T;H^1(\Om))$ be the sequence of solutions of problem \eqref{eq:m5_eps}--\eqref{eq:m8_eps}.
Then, the limit function $v_0\in L^2(0,T;H^1(\Om))$ appearing in \eqref{eq:m8}
is the unique solution of the problem
\begin{equation}\label{eq:m26}
\begin{aligned}
&\mathcal{M}_{\Y}(\fpc^{-1}) \pder{v_{0}}{t} -\Delta v_0+k^2\Theta v_0=f\,,\qquad &\hbox{in $\Om_T$;}
\\
& \pder{v_0}{\nu}=0\,,\qquad & \hbox{on $\partial\Om\times(0,T)$;}
\\
& v_0(x,0)=\frac{1}{\mathcal{M}_\Y(\fpc^{-1})}\overline u\,,\qquad & \hbox{in $\Om$,}
\end{aligned}
\end{equation}
where $\mathcal{M}_{\Y}$ denotes the mean average on $\Y$ and $\Theta$ is the capacity of the inclusion $\hole$, defined by
\begin{equation}\label{eq:m34}
\Theta = \int_{\R^n\setminus\hole}\nabla_z \theta\nabla_z\theta \di z\,,
\end{equation}
where $\theta$ is the capacitary potential, i.e., a function satisfying $\theta\in K_\hole$, $\theta=1$ a.e. in $\hole$ and harmonic in
$\R^{n}\setminus \hole$, that is
\begin{equation}\label{eq:m35}
\int_{\R^n\setminus\hole}\nabla_z\theta\nabla_z \Psi\di z=0\,,\quad \forall \Psi\in K_\hole\ \hbox{s.t. $\Psi=0$ a.e. in $\hole$.}
\end{equation}
\end{theorem}

\begin{proof}
The proof is inspired by the ideas in \cite[Proof of Theorem 3.1]{Cioranescu:Damlamian:Griso:Onofrei:2008}.

\noindent
We use here the convergence $v_\eps\to v_0$ strongly in $L^2(\Om_T)$, according to Remark~\ref{r:v_strong}.
Moreover, by the second convergence in \eqref{eq:m1}, \eqref{eq:m20}, \eqref{eq:m17}, and \eqref{eq:m18}, we obtain the existence of a function $W$ as in Proposition~\ref{p:p1}, such that
\begin{alignat}2
\label{eq:m22}
& \frac{\eta^{n/2-1}}{\eps}\unfoldee(v_\eps)\wto V=W+kv_0, & \hbox{weakly in $L^2(\Om_T;L^2_{loc}(\R^n))$;}
\\
\nonumber
& \frac{\eta^{n/2-1}}{\eps}\nabla_z(\unfoldee(v_\eps))\chi_{\frac{1}{\eta}\Y} &
\\
\label{eq:m25bis}
& \qquad\qquad =\eta^{n/2}\unfoldee(\nabla v_\eps)\chi_{\frac{1}{\eta}\Y}\wto \nabla_z W, &\hbox{weakly in $L^2(\Om_T\times\R^n)$,}
\end{alignat}
where we used also \eqref{eq:m13}.
Notice that $W=V-kv_0$, $\nabla_z W=\nabla_z V$, and,
since $\unfoldee(v_\eps)=0$ a.e. in $\Om_T\times\hole$, then also $V=0$ a.e. in $\Om\times\hole$.

Now, set
\begin{equation}\label{eq:m29}
w_{\eps,\eta}(x)=w_\hole-w\left(\frac{1}{\eta}\left\{\frac{x}{\eps}\right\}_\Y\right),\qquad\hbox{for $x\in\R^n$,}
\end{equation}
where $w\in \mathcal{C}^\infty_c(\R^n)\cap K_\hole$, and $w_\hole$ is the constant value assumed by $w$ on $\hole$.
Notice that $w_{\eps,\eta}=0$ a.e. in $\Oint_{\eps}$. Moreover, as in
\cite[Lemma 3.3]{Cioranescu:Damlamian:Griso:Onofrei:2008},
it follows that $w_{\eps,\eta}\wto w_\hole$, weakly in $H^1(\Om)$, and therefore also strongly in $L^2(\Om)$,
and $\unfoldee(\nabla w_{\eps,\eta})=-\frac{1}{\eps\eta}\nabla_z w$.

Take $\phi(x,t)=r(t)\varphi(x)w_{\eps,\eta}(x)$, with $r\in \mathcal{C}^1([0,T])$ and $r(T)=0$, and $\varphi\in \mathcal{C}^1(\overline\Om)$,
as test function in \eqref{eq:weak_form1}, thus obtaining
\begin{multline}\label{eq:weak_form4}
-\int_{\Om_T}\frac{v_\eps}{\fpc_\eps}r_t\varphi w_{\eps,\eta}\di x\di t+
\int_{\Om_T} 
w_{\eps,\eta}r
\nabla v_\eps\nabla\varphi 
\di x\di t+
\int_{\Om_T} 
\varphi r
\nabla v_\eps\nabla w_{\eps,\eta}
\di x\di t
\\
=\int_{\Om_T} fr\varphi w_{\eps,\eta}\di x\di t +\int_{\Om} \overline{u}r(0)\varphi w_{\eps,\eta}\di x\,.
\end{multline}
We unfold with the operator $\unfoldee$ only the third integral. We get
\begin{multline}\label{eq:weak_form5}
-\int_{\Om_T}\frac{v_\eps}{\fpc_\eps}r_t\varphi w_{\eps,\eta}\di x\di t
+\int_{\Om_T} 
w_{\eps,\eta}r
\nabla v_\eps\nabla\varphi 
\di x\di t
\\
+\eta^n\int_{\Om_T\times\frac{1}{\eta}\Y} 
\unfoldee(\varphi r)
\unfoldee(\nabla v_\eps)
\unfoldee(\nabla w_{\eps,\eta})
\di x\di t\di z
\\
=\int_{\Om_T} fr\varphi w_{\eps,\eta}\di x\di t +\int_{\Om} \overline{u}r(0)\varphi w_{\eps,\eta}\di x + O(\eps)\,.
\end{multline}
Taking into account \eqref{eq:m32} and \eqref{eq:m25bis}, 
for the unfolded term we get
\begin{multline}\label{eq:m31}
-
\lim_{\eps\to 0}
\frac{\eta^{n/2-1}}{\eps}
\int_{\Om_T\times\frac{1}{\eta}\Y} 
\eta^{n/2}
\unfoldee(\varphi r)
\unfoldee(\nabla v_\eps)
\nabla_zw
\di x\di t \di z
\\
=
-k
\int_{\Om_T\times(\R^n\setminus\hole)} 
\varphi r
\nabla_z V\nabla_zw
\di x\di t\di z\,,
\end{multline}
once noticing that
$$
\Vert\unfoldee(\varphi r)-\varphi r\Vert_{L^\infty(\Om_T\times\frac{1}{\eta}\Y)}\leq \const \eps\Vert \nabla\varphi\Vert_{L^\infty(\Om)}
\Vert r\Vert_{L^\infty(0,T)}\,.
$$
Hence, passing to the limit, for $\eps\to 0$, in \eqref{eq:weak_form5}
we get
\begin{multline}\label{eq:weak_form6}
-w_\hole\int_{\Om_T}\mathcal{M}_\Y(\fpc^{-1})v_0r_t\varphi\di x\di t
+ w_{\hole}\int_{\Om_T} 
r
\nabla v_0\nabla\varphi 
\di x\di t
\\
-k\int_{\Om_T\times(\R^n\setminus\hole)} 
\varphi r
\nabla_z V\nabla_zw
\di x\di t\di z
\\
= w_{\hole}\int_{\Om_T} fr\varphi \di x\di t + w_{\hole}\int_{\Om} \overline{u}r(0)\varphi \di x\,.
\end{multline}
By taking $w_\hole=0$, it follows
$$
\int_{\R^n\setminus\hole}\nabla_z V\nabla_zw\di z=0\,,
$$
for a.e. $(x,t)\in \Om_T$ and all $w\in K_\hole$, with $w_\hole=0$.
This implies that $V$ is harmonic in $\R^n\setminus\overline\hole$.
Moreover, on integrating by parts in \eqref{eq:weak_form6}
and then dividing by $w_{\hole}\not=0$, we obtain
\begin{multline}
\label{eq:hom1}
-\int_{\Om_T}\mathcal{M}_\Y(\fpc^{-1})v_0r_t\varphi\di x\di t
+\int_{\Om_T} r\nabla v_0\nabla\varphi \di x\di t
\\
-k\!\!\!\!\!\int_{\Om_T\times\partial\hole}\!\!\!\!\!\!\varphi r\nabla_z V 
\nu_\hole\di x\di t\di \sigma(z)
= \int_{\Om_T} fr\varphi \di x\di t + \int_{\Om} \overline{u}r(0)\varphi \di x\,,
\end{multline}
where $\nu_\hole$ is the inward unit normal to the regular hole $\hole$.
Recalling the definition of the capacitary potential $\theta$ given in \eqref{eq:m35},
by a simple computation, we get
\begin{multline}\label{eq:m9}
\int_{\partial\hole}\nabla_z V \nu_\hole\di \sigma(z)=
\int_{\R^n\setminus\hole}\nabla_z V\nabla_z\theta \di z
=\int_{\R^n\setminus\hole}\nabla_z\theta\nabla_z (V-kv_0) \di z
\\
= -kv_0\int_{\partial\hole}\nabla_z \theta \nu_\hole\di \sigma(z)
=-kv_0\int_{\R^n\setminus\hole}\nabla_z \theta \nabla_z\theta\di z
=-kv_0\Theta\,,
\end{multline}
where we have taken into account that $W=V-kv_0\in K_\hole$.
In order to obtain \eqref{eq:m9}, we have used first $\theta$ as test function for the equation satisfied by $V$,
and then $W$ as test function for the equation satisfied by $\theta$.
Inserting \eqref{eq:m9} in \eqref{eq:hom1} and localizing in $\Om_T$, taking into account the density of the product functions, we get
exactly the weak formulation of the homogenized problem \eqref{eq:m26}.

\noindent
Uniqueness is a direct consequence of the linearity of \eqref{eq:m26}, so that the whole sequence $\{v_\eps\}$, and not only a subsequence,
converges to $v_0$.
\end{proof}

The previous result is, essentially, the parabolic version with homogeneous Neumann boundary condition of the result presented in \cite[Section 3]{Cioranescu:Damlamian:Griso:Onofrei:2008}, which was originally obtained in \cite{Cioranescu:Murat:1982},
for the elliptic case with homogeneous Dirichlet boundary condition.

As a consequence of such a result, we get the homogenized equation for the original Fokker--Planck problem \eqref{eq:pde_2_eps}--\eqref{eq:init_2_eps}, as stated in the following theorem.

\begin{theorem}\label{t:t2}
Let $\{u_\eps\}\subset L^2(\Om_T)$ be the sequence of solutions 
of problems \eqref{eq:pde_2_eps}--\eqref{eq:init_2_eps}
and
\eqref{eq:pde_1_eps}--\eqref{eq:init_1_eps}.
Then, the function $u_0\in L^2(\Om_T)$, appearing in \eqref{eq:m8},
is the unique solution of the problem
\begin{equation}\label{eq:m26bis}
\begin{aligned}
& \pder{u_{0}}{t} -\Delta \left(\frac{1}{\mathcal{M}_\Y(\fpc^{-1})}u_0\right)
+k^2\Theta \frac{1}{\mathcal{M}_\Y(\fpc^{-1})}u_0=f\,,\ &&\hbox{in $\Om_T$;}
\\
& \pder{}{\nu}\left(\frac{1}{\mathcal{M}_\Y(\fpc^{-1})}u_0\right)=0\,,\  && \!\!\!\!\!\!\!\!\!\!\!\!\!\!\!\!\!\!\!\!\!\!\!\!\hbox{on $\partial\Om\times(0,T)$;}
\\
& u_0(x,0)=\overline u\,,\ && \hbox{in $\Om$,}
\end{aligned}
\end{equation}
where
$\mathcal{M}_{\Y}$ and $\Theta$ are defined in Theorem \ref{t:t1}.
\end{theorem}

\begin{proof}
Since, by Remark~\ref{r:v_strong}, we have that $v_\eps\to v_0$ strongly in $L^2(\Om_T)$ and $\fpc_\eps^{-1}\wto \mathcal{M}_{\Y}(\fpc^{-1})$
weakly$^*$ in $L^\infty(\Om)$,
it follows that
\begin{equation}
  \label{eq:t2_i}
  u_{\eps}
  =
  u_{\eps}
  \chi_{\Oout_{\eps,T}}
  +
  u_{\eps}
  \chi_{\Oint_{\eps,T}}
  =
  \fpc_\eps^{-1} v_\eps
  +
  u_{\eps}
  \chi_{\Oint_{\eps,T}}
  \wto
  \mathcal{M}_{\Y}(\fpc^{-1}) v_0
  \,,
\end{equation}
weakly in $L^2(\Om_T)$. Indeed as $\eps\to 0$, in all cases 
where $\eta\to 0$, we have also that
$\abs{\Oint_{\eps}}\to 0$, so that from \eqref{eq:F_1_eps}
we have
\begin{equation}
  \label{eq:homog_deg_intl}
  \norma{\unk_\eps}{L^2(\Oint_{\eps,T})}^{2}
  \le
  2T
  \norma{\initd}{L^2(\Oint_{\eps})}^{2}
  +
  2T^{2}
  \norma{f}{L^2(\Oint_{\eps,T})}^{2}
  \to
  0
  \,.
\end{equation}
However, by \eqref{eq:m8}, it follows that $u_\eps\wto u_0$ weakly in $L^2(\Om_T)$.
Therefore, it is enough to replace $v_0=\big(\mathcal{M}_{\Y}(\fpc^{-1})\big)^{-1}u_0$ in problem \eqref{eq:m26}.
Uniqueness is a standard matter for classical parabolic equations, so that the whole sequence $\{v_\eps\}$, and not only a subsequence,
converges to $v_0$.
\end{proof}

Finally we track the limiting behavior of the total distribution of mass.

\begin{corollary}
  \label{co:homog_deg_total}
  As $\eps\to0$ the solution $u_{\eps}$ satisfies for every $0<t<T$ and every $\varphi\in \mathcal{C}(\overline{\Om})$
  \begin{equation*}
    \int_{\Om}
    \varphi
    [\unk_{\eps}(t)
    \di x
    +
    \di\msr{\fpc_{\eps}u_{\eps}}{t}
    ]
    \to
    \int_{\Om}
    \varphi
    \di \finmass_{0t}
    \,,
  \end{equation*}
  where $\msr{\fpc_{\eps}\unk_{\eps}}{t}$ has been introduced in Remark~\ref{r:deg_mass} and 
  \begin{equation}
    \label{eq:homog_deg_total_nn}
    {}\di \finmass_{0t}
    =
    \Big\{
    u_{0}(x,t)
    +
    k^2\Theta \frac{1}{\mathcal{M}_\Y(\fpc^{-1}(x))}
    \int_{0}^{t}
    u_0(x,\tau)
    \di\tau
    \Big\}
    \di x
    \,.
  \end{equation}
\end{corollary}

\begin{proof}
  By approximation we may assume $\varphi\in \mathcal{C}^{1}(\overline{\Om})$.
  Using definition \eqref{eq:elliptic_flux},
  reasoning as in \eqref{eq:add_limit_iv} and \eqref{eq:add_limit_v},
  and invoking \eqref{eq:F_1_eps}, 
  we arrive at
  \begin{multline}
    \label{eq:homog_deg_total_i}
    \int_{\Om}
    \varphi
    [\unk_{\eps}(t)
    \di x
    +
    \di\msr{\fpc_{\eps}\unk_{\eps}}{t}
    ]
    =
    \int_{\Om}
    u_{\eps}(t)
    \varphi
    \di x
    +
    \int_{\Memb}
    \int_{0}^{t}
    \pder{(\fpc_{\eps} u_{\eps})}{\normint}
    \di \tau
    \varphi
    \di S
    \\
    =
    \int_{\Om}
    F(t)
    \varphi
    \di x
    -
    \int_{0}^{t}
    \int_{\Oout_{\eps}}
    \grad (\fpc_{\eps} u_{\eps})
    \grad \varphi
    \di x
    \di\tau
    \,.
  \end{multline}
  As $\eps\to 0$ we have that, from \eqref{eq:m8}, 
  the right hand side of \eqref{eq:homog_deg_total_i} approaches
 \begin{displaymath}
    \int_{\Om}
    F(t)
    \varphi
    \di x
    -
    \int_{0}^{t}
    \int_{\Om}
    \grad v_{0}
    \grad \varphi
    \di x
    \di\tau
 .
 \end{displaymath}
 Moreover, invoking \eqref{eq:m26} or \eqref{eq:m26bis},
 we get 
  \begin{multline}
    \label{eq:homog_deg_total_ii}
    \int_{\Om}
    F(t)
    \varphi
    \di x
    -
    \int_{0}^{t}
    \int_{\Om}
    \grad v_{0}
    \grad \varphi
    \di x
    \di\tau
    \\
    =
    \int_{\Om}
    \Big\{
    u_{0}(t)
    +
    k^2\Theta \frac{1}{\mathcal{M}_\Y(\fpc^{-1})}
    \int_{0}^{t}
    u_0(x,\tau)
    \di\tau
    \Big\}
    \varphi
    \di x
    \,.
  \end{multline}
\end{proof}

\subsection{The limit equation in the case $\eta\gg\eps^{2/(n-2)}$}\label{ss:superd}
We assume here that
\begin{equation}
  \label{eq:superd_def}
  \lim_{\eps\to 0}
  \frac{\eta^{\frac{n}{2}-1}}{\eps}
  =
  +
  \infty
  \,.
\end{equation}

In this case, as it will be detailed in the 
next theorem, the inclusions $\Oint_{\eps}$ tend, in the limit $\eps\to0$,
to spread over the whole domain $\Om$. In other words, 
the function $v_\eps\to0$ so that the total mass is 
represented by the limit of the internal problem \eqref{eq:F_1_eps}
and by the external mass that concentrates on the boundary of the inclusions.

\begin{theorem}
  \label{t:superd}
  Under assumption \eqref{eq:superd_def}, as $\eps\to0$ we have 
  that $v_{\eps}\to 0$
  strongly in $L^{2}(\OsetT)$ and that
  the solution $u_{\eps}$ satisfies for every $0<t<T$ and every $\varphi\in \mathcal{C}(\overline{\Om})$
  \begin{equation*}
    \int_{\Om}
    \varphi
    [\unk_{\eps}(t)
    \di x
    +
    \di\msr{\fpc_{\eps}\unk_{\eps}}{t}
    ]
    \to
    \int_{\Om}
    F(t)
    \varphi
    \di x
    \,,
  \end{equation*}
  where $\msr{\fpc_{\eps}\unk_{\eps}}{t}$ has been 
introduced in Remark~\ref{r:deg_mass} and $F$ is defined in \eqref{eq:add_F}.
  Therefore, the density $F$ of the limiting measure satisfies in 
  the standard weak sense
  \begin{equation}
    \label{eq:superd_nn}
    \pder{F}{t}
    =
    f
    \,,
    \quad
    \text{in $\OsetT$;}
    \qquad
    F(x,0)
    =
    \initd(x)
    \,,
    \quad
    x\in\Oset
    \,.
  \end{equation}
\end{theorem}

\begin{proof}
  Recalling \cite[Corollary~4.5.3]{Ziemer:WDF} and the scaling properties 
  (in the parameter $\eta$) of capacity applied to the inclusion $B$,  
  one obtains for any cell $\Y_{\eps}$, on setting $\hat v(y)=v_{\eps}(x_{c,\eps}+\eps y)$, where $x_{c,\eps}$ is the center of $\Y_{\eps}$,
  \begin{equation}
    \label{eq:superd_cap}
    \begin{split}
      \int_{\Y_{\eps}}
      v_{\eps}(x)^{2}
      \di x
      &=
      \eps^{n}
      \int_{\Y}
      \hat v(y)^{2}
      \di y
      \le
      \eps^{n}
      \Big(
      \int_{\Y}
      \hat v(y)^{2^{*}}
      \di y
      \Big)^{\frac{n-2}{n}}
      \\
      &\le
      \eps^{n}
      \frac{\gamma}{\eta^{n-2}}
      \int_{\Y}
      \abs{\grad_{y}\hat v(y)}^{2}
      \di y
      =
      \gamma
      \frac{\eps^{2}}{\eta^{n-2}}
      \int_{\Y_{\eps}}
      \abs{\grad_{x} v_{\eps}(x)}^{2}
      \di x
      \,,
    \end{split}
  \end{equation}
  where $\gamma=\gamma(B,n)$.
  On summing on the cells, we easily obtain from \eqref{eq:energy2} and \eqref{eq:superd_def} that $v_{\eps}\to 0$ in $L^{2}(\OsetT)$.

  Next we reason as in the proof of Corollary~\ref{co:homog_deg_total}, up to \eqref{eq:homog_deg_total_i}. Here we simply note that the last integral there vanishes as $\eps\to 0$ since $\grad v_{\eps}\wto 0$ in $L^{2}(\OsetT)$ owing to the convergence of $v_{\eps}$ to $0$.
\end{proof}

\begin{remark}
\label{r:comparison}
Under the assumption \eqref{eq:superd_def}, in the case $\eta(\eps)\to0$
for $\eps\to0$,
we have that $|\Oint_{\eps}|\to0$ in the same limit, 
see \eqref{eq:homog_deg_intl}.
Moreover, from Theorem~\ref{t:superd}, we have that $v_\eps\to0$.
Then
\begin{equation*}
\|u_\eps\|^2_{L^2(\Om_T)}
=
\|u_\eps\|^2_{L^2(\Oint_{\eps,T})}
+
\|u_\eps\|^2_{L^2(\Oout_{\eps,T})}
\to0,\quad\textup{ as }\eps\to0.
\end{equation*}
Indeed,
the first term tends to zero, as proven in \eqref{eq:homog_deg_intl}, and
the second term is bounded by $C\|v_e\|^2_{L^2(\Om_T)}$
which tends to zero, as well.
\end{remark}

\subsection{The limit equation in the case $\eta\ll\eps^{2/(n-2)}$}\label{ss:sottod}

We assume that $\eta=\eta(\eps)$ satisfies
\begin{equation}\label{eq:m32bis}
\lim_{\eps\to 0}\frac{\eta^{n/2-1}}{\eps}=0.
\end{equation}

\begin{theorem}\label{t:t1_supcrit}
Let $\{v_\eps\}\subset L^2(0,T;H^1(\Om))$ be the sequence of solutions of problem \eqref{eq:m5_eps}--\eqref{eq:m8_eps}.
Then, the limit function $v_0\in L^2(0,T;H^1(\Om))$ appearing in \eqref{eq:m8}
is the unique solution of the problem
\begin{equation}\label{eq:a41}
\begin{aligned}
&\mathcal{M}_{\Y}(\fpc^{-1}) 
\pder{v_{0}}{t} -\Delta v_0=f\,,\qquad &\hbox{in $\Om_T$;}
\\
& \pder{v_0}{\nu}=0\,,\qquad & \hbox{on $\partial\Om\times(0,T)$;}
\\
& v_0(x,0)=\frac{1}{\mathcal{M}_\Y(\fpc^{-1})}\overline u\,,\qquad & \hbox{in $\Om$,}
\end{aligned}
\end{equation}
where $\mathcal{M}_{\Y}$ is defined in Theorem \ref{t:t1}.
\end{theorem}

\begin{proof}
We can proceed as in the proof of Theorem \ref{t:t1}, taking into account that \eqref{eq:m25bis} is still in force.
Then, taking as test function in \eqref{eq:weak_form1} $\phi(x,t)=r(t)\varphi(x)w_{\eps,\eta}(x)$, with $r\in \mathcal{C}^1([0,T])$ and $r(T)=0$,  $\varphi\in \mathcal{C}^1(\overline\Om)$, and $w_{\eps,\eta}$ as in \eqref{eq:m29}, unfolding and passing to the limit for $\eps\to 0$,
we arrive at \eqref{eq:weak_form5}. However, in the present case, we also have
\begin{equation}\label{eq:m42}
\lim_{\eps\to 0}
\frac{\eta^{n/2-1}}{\eps}
\int_{\Om_T\times\frac{1}{\eta}\Y} 
\eta^{n/2}
\unfoldee(\varphi r)
\unfoldee(\nabla v_\eps)
\nabla_zw
\di x\di t \di z
=0\,,
\end{equation}
due to \eqref{eq:m32bis}. Hence, we get
\begin{multline}\label{eq:a43}
-w_\hole\int_{\Om_T}\mathcal{M}_\Y(\fpc^{-1})v_0r_t\varphi\di x\di t
+ w_{\hole}\int_{\Om_T} r\nabla v_0\nabla\varphi \di x\di t
\\
= w_{\hole}\int_{\Om_T} fr\varphi \di x\di t + w_{\hole}\int_{\Om} \overline{u}r(0)\varphi \di x\,,
\end{multline}
which, after dividing by $ w_{\hole}$ and taking into account the density of the product functions, gives the weak formulation of \eqref{eq:a41}.
\noindent
Uniqueness is a standard matter for classical parabolic equations, so that the whole sequence $\{v_\eps\}$, and not only a subsequence,
converges to $v_0$.
\end{proof}

The previous result is in accordance with the elliptic version for Dirichlet homogeneous boundary conditions
presented in \cite[Section 3]{Cioranescu:Damlamian:Griso:Onofrei:2008} and originally obtained in \cite{Cioranescu:Murat:1982}.

The homogenized equation for the original Fokker--Planck problem
\eqref{eq:pde_2_eps}--\eqref{eq:init_2_eps} is given in the following
theorem.

\begin{theorem}\label{t:t2_supcrit}
Let $\{u_\eps\}\subset L^2(\Om_T)$ be the sequence of solutions of 
problems \eqref{eq:pde_2_eps}--\eqref{eq:init_2_eps} 
and \eqref{eq:pde_1_eps}--\eqref{eq:init_1_eps}.
Then, the function $u_0\in L^2(\Om_T)$, appearing in \eqref{eq:m8},
is the unique solution of the problem
\begin{equation}\label{eq:m26ter}
\begin{aligned}
& \pder{u_{0}}{t} -\Delta \left(\frac{1}{\mathcal{M}_\Y(\fpc^{-1})}u_0\right)
=f\,,\ &&\hbox{in $\Om_T$;}
\\
& \pder{}{\nu}\left(\frac{1}{\mathcal{M}_\Y(\fpc^{-1})}u_0\right)=0\,,\  && 
\hbox{on $\partial\Om\times(0,T)$;}
\\
& u_0(x,0)=\overline u\,,\ && \hbox{in $\Om$,}
\end{aligned}
\end{equation}
where $\mathcal{M}_{\Y}$ is defined in Theorem \ref{t:t1}.
\end{theorem}

\begin{proof}
  Notice that $v_\eps\to v_0$ strongly in $L^2(\Om_T)$, by Remark~\ref{r:v_strong}, $\fpc_\eps^{-1}\wto \mathcal{M}_{\Y}(\fpc^{-1})$
  weakly$^*$ in $L^\infty(\Om)$, and $u_\eps\wto u_0$ weakly in $L^2(\Om_T)$ by \eqref{eq:m8}. Reasoning also as in \eqref{eq:homog_deg_intl} we can easily obtain \eqref{eq:m26ter}
  from \eqref{eq:a41} by replacing $v_{0}=u_{0}/\mathcal{M}_{\Y}(\fpc^{-1})$.
\end{proof}

\begin{remark}
  \label{r:sottod_mass}
  Under the assumption \eqref{eq:m32bis}, a version of Corollary~\ref{co:homog_deg_total}, where we let formally $k=0$ in the statement, follows essentially with the same proof.
\end{remark}

\section{Homogenization of the non degenerate problem}
\label{ss:homog_ndeg}

Here 
we are interested in homogenizing the
problem \eqref{eq:pde_d}--\eqref{eq:init_d}
for $\delta$ fixed.
For the reader's convenience, we rewrite it 
omitting the subscript index $\delta$ from the 
notation of the unknown:
\begin{alignat}{2}
  \label{eq:pde_22_eps}
  \pder{\unk_\eps}{t}
  -
  \Lapl(\fpc_{\eps,\delta}\unk_\eps)
  &=
  f
  \,,
  &\qquad&
  \text{in $\Om_{T}$;}
  \\
  \label{eq:bdr_22_eps}
  \pder{(\fpc_{\eps,\delta}\unk_\eps)}{\normal}
  &=
  0
  \,,
  &\qquad&
  \text{on $\bdr{\Oset}\times(0,T)$;}
  \\
  \label{eq:init_22_eps}
  \unk_\eps(x,0)
  &=
  \initd(x)
  \,,
  &\qquad&
  \text{in $\Oset$.}
\end{alignat}
Here, $\fpc_{\eps,\delta}$ is the coefficient defined in \eqref{eq:m3}.

As above, if we set $v_\eps=\fpc_{\eps,\delta} u_\eps$, we can 
rewrite the previous problem as
\begin{alignat}{2}
  \label{eq:m52_eps}
  \frac{1}{\fpc_{\eps,\delta}}\pder{v_\eps}{t}
  -
  \Lapl v_\eps
  &=
  f
  \,,
  &\qquad&
  \text{in $\Oset_{T}$;}
  \\
  \label{eq:m62_eps}
  \pder{v_\eps}{\normal}
  &=
  0
  \,,
  &\qquad&
  \text{on $\bdr{\Oset}\times(0,T)$;}
  \\
  \label{eq:m82_eps}
  v_\eps(x,0)
  &=
  \overline v_{\eps}(x)
  \,,
  &\qquad&
  \text{in $\Oset$,}
\end{alignat}
where $\overline v_{\eps}=\fpc_{\eps,\delta}\initd\in L^2(\Om)$.

On invoking \eqref{eq:energy_v} we obtain that, up to a
subsequence, in the limit $\eps\to0$
\begin{equation}\label{eq:m88}
\begin{aligned}
& v_\eps\to v_0 &&\qquad \hbox{strongly in $L^2(\Om_T)$;}\\
& \nabla v_\eps\wto \nabla v_0 &&\qquad \hbox{weakly in $L^2(\Om_T)$,}
\end{aligned}
\end{equation}
and, since $\fpc_{\eps,\delta}\ge C\delta>0$ for all $\eps>0$,
\begin{equation}
  \label{eq:m88_bis}
  u_\eps\wto u_0 \qquad \hbox{weakly in $L^2(\Om_T)$,}
\end{equation}
for a suitable  $u_0\in L^2(\Om_T)$.
Note that both $u_0$ and $v_0$ depend on the fixed parameter 
$\delta$. 

\subsection{The limit equation in the case $\eta\to 0$}\label{ss:critical_ndeg}

We assume that $\eta=\eta(\eps)$, with $\eta(\eps)$ being a 
general infinitesimal function, for $\eps\to 0$.

\begin{theorem}\label{t:t12}
Let $\{v_\eps\}\subset L^2(0,T;H^1(\Om))$ be the sequence of solutions 
of problem \eqref{eq:m52_eps}--\eqref{eq:m82_eps}
and assume that $\eta(\eps)\to0$ as $\eps\to0$.
Then, the function $v_0\in L^2(0,T;H^1(\Om))$, appearing in \eqref{eq:m88},
is the unique solution of the problem
\eqref{eq:a41} and thus it does not depend on $\delta$. 
\end{theorem}

\begin{proof}
Take $\phi(x,t)=r(t)\varphi(x)$, with $r\in \mathcal{C}^1([0,T])$, $r(T)=0$,
and $\varphi\in \mathcal{C}^1(\overline\Om)$, as 
test function in \eqref{eq:weak_form12}, thus obtaining
\begin{multline}\label{eq:weak_form42}
-\int_{\Om_T}\frac{v_\eps}{\fpc_{\eps,\delta}}r_t\varphi \di x\di t+
\int_{\Om_T} r \nabla v_\eps\nabla\varphi \di x\di t
\\
=\int_{\Om_T} fr\varphi \di x\di t +\int_{\Om} \overline{u}r(0)\varphi \di x\,.
\end{multline}
Notice that
$$
\int_{\Om_T}\frac{v_\eps}{\fpc_{\eps,\delta}}r_t\varphi \di x\di t=
\int_{\Om_T}v_\eps r_t\varphi\left(\frac{1}{\delta\fpc_\eps}\chi_{\Oint_{\eps}}+\frac{1}{\fpc_\eps}\chi_{\Oout_{\eps}}\right) \di x\di t
=I_{1,\eps}+I_{2,\eps}\,.
$$
By \eqref{eq:m7}, we obtain
$$
|I_{1,\eps}|\leq \frac{\const}{\delta}\Vert v_\eps\Vert_{L^2(\Om_T)}\sqrt{|\Oint_\eps|}\leq \const\eta^{n/2}\to 0\,,
\quad\hbox{for $\eta\to 0$,}
$$
where we have taken into account \eqref{eq:energy3} and the fact that $|\Oint_\eps|\sim \frac{|\Oset|}{\eps^n}\eta^n\eps^n=|\Oset|\eta^n$.
By the same argument, $\chi_{\Oout_\eps}\to 1$ 
strongly in $L^q(\Oset)$, for any $q\geq 1$, so that
$$
\frac{1}{\fpc_\eps}\chi_{\Oout_\eps}\wto \mathcal{M}_\Y(\fpc^{-1})\,,\qquad \hbox{weakly in $L^2(\Om)$.}
$$
Hence,
$$
I_{2,\eps}\to \int_{\Om_T}v_0r_t\varphi\mathcal{M}_\Y(\fpc^{-1})\di x\di t\,.
$$

Note that we also have
\begin{equation}
\label{eq:star1}
\frac{1}{\delta\fpc_\eps}
=
\frac{1}{\delta\fpc_\eps}\chi_{\Oint_{\eps}}
+\frac{1}{\fpc_\eps}\chi_{\Oout_{\eps}}
\wto
\mathcal{M}_{\Y}(\fpc^{-1}),
\quad
\textup{ weakly in } L^2(\Om).
\end{equation}
Therefore, passing to the limit for $\eps\to 0$, we get
\begin{multline}\label{eq:weak_form62}
-\int_{\Om_T}\mathcal{M}_\Y(\fpc^{-1})v_0r_t\varphi\di x\di t
+ \int_{\Om_T} \nabla v_0\nabla\varphi r\di x\di t
\\
= \int_{\Om_T} fr\varphi \di x\di t 
+ \int_{\Om} \overline{u}r(0)\varphi \di x\,,
\end{multline}
which is the weak formulation of the problem \eqref{eq:a41}. Again, 
uniqueness follows by the linearity of the homogenized problem,
so that the whole sequence, and not only a subsequence, converges to $v_0$.
\end{proof}

As a consequence, we get the homogenized equation for 
the original Fokker--Planck problem 
\eqref{eq:pde_22_eps}--\eqref{eq:init_22_eps}, 
as stated in the following theorem.

\begin{theorem}\label{t:t22}
Let $\{u_\eps\}\subset L^2(\Om_T)$ be the sequence of solutions 
of problem \eqref{eq:pde_22_eps}--\eqref{eq:init_22_eps}
and assume that $\eta(\eps)\to0$ and $\eps\to0$.
Then, the function $u_0\in L^2(\Om_T)$, appearing in \eqref{eq:m88_bis},
is the unique solution of the problem \eqref{eq:m26ter} and thus 
it does not depend on $\delta$.
\end{theorem}

\begin{proof}
Recalling that $u_\eps=v_\eps/(\delta\fpc_\eps)$
and
using \eqref{eq:m88}, \eqref{eq:m88_bis}, and \eqref{eq:star1}, 
we get
$u_\eps\wto u_0$
and 
$u_\eps\wto\mathcal{M}_\Y(\fpc^{-1})v_0$, 
which yields 
$v_0=u_0/\mathcal{M}_\Y(\fpc^{-1})$.
Thus the statement follows from Theorem~\ref{t:t12}.
\end{proof}

\subsection{The limit equation in the case $\eta= 1$}\label{ss:superndeg}

\begin{theorem}\label{t:t13}
Let $\{v_\eps\}\subset L^2(0,T;H^1(\Om))$ be the sequence of 
solutions of problem \eqref{eq:m52_eps}--\eqref{eq:m82_eps}
and assume $\eta=1$.
Then, the function $v_0\in L^2(0,T;H^1(\Om))$, appearing in \eqref{eq:m88},
is the unique solution of the problem
\begin{equation}\label{eq:m263}
\begin{aligned}
&\mathcal{M}_{\Y}(\fpc_{\delta}^{-1}) \pder{v_{0}}{t} -\Delta v_0=f\,,\qquad &\hbox{in $\Om_T$;}
\\
& \pder{v_0}{\nu}=0\,,\qquad & \hbox{on $\partial\Om\times(0,T)$;}
\\
& v_0(x,0)=\frac{1}{\mathcal{M}_\Y(\fpc_{\delta}^{-1})}\overline u\,,\qquad & \hbox{in $\Om$,}
\end{aligned}
\end{equation}
where
$$
\mathcal{M}_{\Y}(\fpc_{\delta}^{-1})(x)=\int_{\hole}\frac{1}{\delta\fpc(x,y)}\di y+\int_{\Y^*}\frac{1}{\fpc(x,y)}\di y\,.
$$
\end{theorem}

\begin{proof}
The proof can be carried out as in the case of Theorem \ref{t:t12}, the only difference being in the term
$$
I_\eps:=\int_{\Om_T}\frac{v_\eps}{\fpc_{\eps,\delta}}r_t\varphi \di x\di t=
\int_{\Om_T}v_\eps r_t\varphi\left(\frac{1}{\delta\fpc_\eps}\chi_{\Oint_{\eps}}+\frac{1}{\fpc_\eps}\chi_{\Oout_{\eps}}\right) \di x\di t,
$$
which can be treated passing to the standard 
unfolding operator since, in this case, 
the inclusions $\Omega^1_\eps$ rescale periodically with
respect to $\eps$.
We get
\begin{multline*}
I_\eps=\int_{\Om_T\times\Y}\unfolde(v_\eps)\unfolde( r_t\varphi)\left(\unfolde\left(\frac{1}{\delta\fpc_\eps}\right)\unfolde(\chi_{\Oint_{\eps}})
\right.
\\
\left.
+\unfolde\left(\frac{1}{\fpc_\eps}\right)\unfolde(\chi_{\Oout_{\eps}})\right) \di y\di x\di t +O(\eps).
\end{multline*}
By passing to the limit and taking into account that
\begin{equation}
  \label{eq:t13_i}
  \unfolde\left(\frac{1}{\delta\fpc_\eps}\right)\unfolde(\chi_{\Oint_{\eps}})\to \frac{1}{\delta\fpc}\chi_{\hole}
  \quad\hbox{and}\quad
  \unfolde\left(\frac{1}{\fpc_\eps}\right)\unfolde(\chi_{\Oout_{\eps}})\to \frac{1}{\fpc}\chi_{\Y^*}
\end{equation}
strongly in $L^2(\Om_T\times\Y)$, we obtain
$$
I_\eps\to \int_{\Om_T}v_0 r_t\varphi\left(\int_{\hole}\frac{1}{\delta\fpc(x,y)}\di y+\int_{\Y^*}\frac{1}{\fpc(x,y)}\di y\right)\di x\di t\,,
$$
which gives the thesis.
\end{proof}

Passing to the homogenized equation for the original 
Fokker--Planck problem \eqref{eq:pde_22_eps}--\eqref{eq:init_22_eps},
we obtain the following result.

\begin{theorem}\label{t:t23}
Let $\{u_\eps\}\subset L^2(\Om_T)$ be the sequence of solutions 
of problem \eqref{eq:pde_22_eps}--\eqref{eq:init_22_eps}
and assume $\eta=1$.
Then, the function $u_0\in L^2(\Om_T)$, appearing in \eqref{eq:m88},
is the unique solution of the problem
\begin{equation}\label{eq:a46}
\begin{aligned}
& \pder{u_{0}}{t} -\Delta \left(\frac{1}{\mathcal{M}_\Y(\fpc^{-1}_{\delta})}\,
u_0\right)=f, &\!\!\!\!\!\hbox{in $\Om_T$;}
\\
& \pder{}{\nu}\left(\frac{1}{\mathcal{M}_\Y(\fpc^{-1}_{\delta})}u_0\right)=0,&
 \hbox{on $\partial\Om\times(0,T)$;}
\\
& u_0(x,0)=\overline u, & \!\!\hbox{in $\Om$.}
\end{aligned}
\end{equation}
\end{theorem}

\begin{proof}
Recalling that $u_{\eps}=v_{\eps}/\fpc_{\eps,\delta}$, 
similarly as in the proof of Theorem~\ref{t:t22},
thanks to \eqref{eq:t13_i} we obtain
\begin{equation}
  \label{eq:t23_i}
  u_{0}
  =
  \frac{v_{0}}{\mathcal{M}_\Y(\fpc^{-1}_{\delta})}
  \,.
\end{equation}
Then the statement follows by replacing \eqref{eq:t23_i}
in \eqref{eq:m263}.
\end{proof}

We note that the functions $u_0$ and $v_0$ appearing in 
Theorems~\ref{t:t13} and \ref{t:t23} do depend 
on the parameter $\delta$, even if, as said at the 
beginning of this section, this dependence is not explicitly 
reported in the notation. 

\subsection{The limit $\delta\to 0$ of the homogenized problem}\label{s:deghom}

The next step is to let $\delta\to 0$, in the only case where the
homogenized problem depends on $\delta$, i.e., when $\eta=1$. To
this purpose, we first notice that \eqref{eq:m263} leads to the energy
estimate
\begin{equation}\label{eq:a44}
\sup_{t\in(0,T)}\frac{1}{\delta}\int_{\Om}v^2_0\di x+\int_{\Om_T} |\nabla v_0|^2\di x\di t
\leq\const (\Vert f\Vert^2_{L^2(\Om_T)}+\Vert \initd\Vert^2_{L^2(\Om)}),
\end{equation}
where $\const>0$ is independent of $\delta$.
In particular, it follows that
\begin{equation}\label{eq:a45}
\sup_{t\in(0,T)}\int_{\Om}v^2_0\di x\leq\const\delta.
\end{equation}
Therefore, from \eqref{eq:a44} and \eqref{eq:a45}, 
we obtain that $v_0$ tends to $0$ weakly 
in 
$L^2(0,T;H^1(\Om))$
and strongly in
$L^2(\Om_T)$.

On the other hand, concerning the solution $u_0$ of the homogenized 
Fokker--Planck problem \eqref{eq:a46},
we have the following result.

\begin{theorem}\label{t:t10}
Let $u_0$ be the solution of problem \eqref{eq:a46}. Then, 
we have that $u_0\di x\di t\wto F\di x\di t$ in the weak$^*$ sense of measures,
where $F$ is given in \eqref{eq:add_F}, for $\delta\to0$
and thus \eqref{eq:superd_nn} is in force. 
\end{theorem}

\begin{proof}
  On one hand we know from the estimate given by
  Lemma~\ref{l:mass_con} that $u_{0}$ converges in the weak$^{*}$
  sense (up to subsequences). On the other hand, passing to the limit
  in the weak formulation of \eqref{eq:a46}, where we select the test
  function as in \eqref{eq:weaktest}, we obtain
\begin{multline}\label{eq:a47}
  \lim_{\delta\to 0}
  \int_{\OsetT}
  u_{0}
  \varphi
  \di x
  \di t
  =
  -
  \lim_{\delta\to 0}
  \int_{\OsetT} u_0\phi_t\di x  \di t
  \\
  =
  \lim_{\delta\to 0}
  \left(-\int_{\OsetT}\nabla {\left( \frac{1}{|\Y^*|\mathcal{M}_{\Y^*}(\fpc^{-1})+\delta^{-1}|\hole|\mathcal{M}_{\hole}(\fpc^{-1})}u_0\right)}\nabla\phi\di x \di t
  \right.
  \\
  \left.+
    \int_{\OsetT} f\phi\di x  \di t+\int_{\Oset} \initd\phi(0)\di x \right)
  \\
  =
  \lim_{\delta\to 0}\left(-\int_{\OsetT}\nabla v_0\nabla\phi\di x \di t
    +
    \int_{\OsetT} f\phi\di x  \di t+\int_{\Oset} \initd\phi(0)\di x \right)
  \\
  =
  \int_{\OsetT}
  \varphi(x,\tau)
  \Big[
  \int_{0}^{\tau}
  f(x,t)
  \di t
  +
  \initd(x)
  \Big]
  \di x
  \di \tau
  \,.
\end{multline}
This implies the claim.
\end{proof}

\begin{remark}\label{r:r1}
In Sections~\ref{s:fp_nondeg} and \ref{ss:homog_deg} we have first 
computed the degeneration limit $\delta\to0$ and afterwards 
the homogenization limit $\eps\to0$ of the original 
problem \eqref{eq:pde_d}--\eqref{eq:init_d}. 
On the contrary, in Section~\ref{ss:homog_ndeg} we have 
performed the two limits in the reversed order, first the 
homogenization and afterwards the degeneration one. It is natural 
to compare the results and look for possible commutation properties.

As we have already noted, in the homogenized limit problem 
of Theorem~\ref{t:t22}, namely, for $\eta(\eps)\to0$ as $\eps\to0$,
no dependence on the degeneration parameter $\delta$ appears,
so that the resulting problem cannot degenerate.

In particular, comparing the results of Section~\ref{ss:homog_deg} 
with 
Theorems~\ref{t:t22} and \ref{t:t10},
we can distinguish two cases:

i) $\eta\to0$:
the limits $\eps\to 0$ and $\delta\to 0$ for 
problem \eqref{eq:pde_d}--\eqref{eq:init_d} do not commute 
in the 
critical case
$\eta\approx\eps^{2/(n-2)}$
of Section~\ref{s:crit} 
and 
in the 
supercritical case
$\eta\gg\eps^{2/(n-2)}$
of Section~\ref{ss:superd}, 
while they do commute in the subcritical case 
$\eta\ll\eps^{2/(n-2)}$
of Section~\ref{ss:sottod};

ii) $\eta=1$: we are then, again,  
in the supercritical case 
$\eta\gg\eps^{2/(n-2)}$
of Subsection~\ref{ss:superd} and the two limits commute.

It is worth noting that in the critical case 
$\eta\approx\eps^{2/(n-2)}$
the 
``terme \'{e}trange venu d'ailleurs" of 
\eqref{eq:m26bis}, already found in 
\cite{Cioranescu:Murat:1982} for the elliptic problem, appears
only if the degeneration limit is taken before the homogenization 
one. In the reverse case the more standard \eqref{eq:m26ter} 
problem
is found. 

In view of these results, a natural question arise about the 
behavior of the model when the degeneration and the homogenization limits 
are taken simultaneously. 
\end{remark}

\section{An explicit solution and a counterexample}
\label{s:od}
In this section we exhibit an explicit solution of the 
one--dimensional 
Fokker--Planck equation which will enable us to build a 
counterexample in which the solution becomes unbounded in a finite time,
though the Fokker--Planck coefficient is bounded away from zero, 
but depends on time. 

We look first at the one-dimensional problem in $\R$ 
\begin{alignat}{2}
  \label{eq:od_pde}
  \odu_{i,t}
  -
  \odfpc_{i}
  \odu_{i,xx}
  &=
  0
  \,,
  &\qquad&
  (-1)^{i}x>0
  \,,
  t>0
  \,,
  \\
  \label{eq:od_dir}
  \odfpc_{1}
  \odu_{1}(0-,t)
  &=
  \odfpc_{2}
  \odu_{2}(0+,t)
  \,,
  &\qquad&
  t>0
  \,,
  \\
  \label{eq:od_flux}
  \odfpc_{1}
  \odu_{1,x}(0-,t)
  &=
  \odfpc_{2}
  \odu_{2,x}(0+,t)
  \,,
  &\qquad&
  t>0
  \,,
  \\
  \label{eq:od_initd}
  \odu_{i}(x,0)
  &=
  \odinit
  \,,
  &\qquad&
  (-1)^{i}x>0
  \,,
\end{alignat}
where $i=1,2$,
and $\odinit$, $\odfpc_{i}>0$ are constants. 
Note that \eqref{eq:od_dir} and \eqref{eq:od_flux}
correspond to \eqref{eq:jump_d} and \eqref{eq:jumpflux_d}.
Below the initial data will be replaced with a bounded piecewise continuous function, and the coefficients $\odfpc_{i}$ with a piecewise constant function depending on $(x,t)$. The definition of weak solution to such problems is then essentially the same as \eqref{eq:weak_form22}, since in this instance the dependence of the coefficient $\fpcde$ on time does not play any role (see, also, the comment at the end of 
the Remark~\ref{r:supbound}); it does have anyway serious implications as we will show presently. Note that, owing to classical results of local regularity, the solution is smooth where the coefficients and data are smooth. Also, we remark that we work with solutions defined in $\R$ for the sake of formal simplicity (to avoid the irrelevant influence of boundary conditions), but our argument is essentially local.


\begin{lemma}
  \label{l:od_mod}
  There exists a solution to \eqref{eq:od_pde}--\eqref{eq:od_initd} satisfying
  \begin{equation}
    \label{eq:od_mod_n}
    \odu_{1}(0-,t)
    =
    \odinit
    \frac{\sqrt{\odfpc_{2}}}{\sqrt{\odfpc_{1}}}
    \,,
    \qquad
    \odu_{2}(0+,t)
    =
    \odinit
    \frac{\sqrt{\odfpc_{1}}}{\sqrt{\odfpc_{2}}}
    \,,
    \qquad
    t>0
    \,.
  \end{equation}
\end{lemma}

\begin{proof}
  Following the classical parabolic theory, 
  see, e.g., \cite[Chapter~4]{LSU},
  we represent the solution with the standard double layer potential
  \begin{equation}
    \label{eq:od_mod_i}
    \odu_{i}(x,t)
    =
    \odinit
    +
    \int_{0}^{t}
    \oddens_{i}(\tau)
    \odfund_{i,x}(x,t-\tau)
    \di\tau
    \,,
    \qquad
    (-1)^{i}x>0
    \,,
    t>0
    \,.
  \end{equation}
  Here $\odfund_{i}$ is the fundamental solution of the heat equation written for diffusivity $\odfpc_{i}$. The first condition on the unknowns $\oddens_{i}$ follows from the jump property of the potential
  \begin{equation}
    \label{eq:od_mod_jump}
    \lim_{(-1)^{i}x\to 0+}
    \odu_{i}(x,t)
    =
    \odinit
    +
    \frac{(-1)^{i+1}}{2\odfpc_{i}}
    \oddens_{i}(t)
    \,,
    \qquad
    i=1\,,2
    \,,
  \end{equation}
  and from \eqref{eq:od_dir}, yielding
  \begin{equation}
    \label{eq:od_mod_ii}
    \odfpc_{1}
    \Big(
    \odinit
    +
    \frac{1}{2\odfpc_{1}}
    \oddens_{1}(t)
    \Big)
    =
    \odfpc_{2}
    \Big(
    \odinit
    -
    \frac{1}{2\odfpc_{2}}
    \oddens_{2}(t)
    \Big)
    \,.
  \end{equation}
  Then according to a classical argument and by exploiting $\odfpc_{i}\odfund_{i,xx}=-\odfund_{i,\tau}$, we differentiate in $x$ and obtain by integration by parts
  \begin{equation}
    \label{eq:od_mod_jjj}
    \odu_{i,x}(x,t)
    =
    \frac{\oddens_{i}(0)}{\odfpc_{i}}
    \odfund_{i}(x,t)
    +
    \int_{0}^{t}
    \frac{\oddens_{i}'(\tau)}{\odfpc_{i}}
    \odfund_{i}(x,t-\tau)
    \di\tau
    \,.
  \end{equation}
  The single layer potential in \eqref{eq:od_mod_jjj} is continuous up to $x=0$, yielding for example
  \begin{equation}
    \label{eq:od_mod_iii}
    \odu_{1,x}(0-,t)
    =
    \frac{\oddens_{1}(0)}{2\odfpc_{1}\sqrt{\pi\odfpc_{1} t}}
    +
    \int_{0}^{t}
    \frac{\oddens_{1}'(\tau)}{2\odfpc_{1}\sqrt{\pi\odfpc_{1}(t-\tau)}}
    \di\tau
    \,,
  \end{equation}
  which, again by the classical theory of integral equations, can be recast as the Abel equation of first kind
  \begin{equation}
    \label{eq:od_mod_iv}
    \frac{2}{\sqrt{\pi}}
    \int_{0}^{t}
    \frac{\odfpc_{1}\odu_{1,x}(0-,\tau)}{\sqrt{t-\tau}}
    \di\tau
    =
    \frac{\oddens_{1}(t)}{\sqrt{\odfpc_{1}}}
    \,,
    \qquad
    t>0
    \,.
  \end{equation}
  From \eqref{eq:od_mod_iv} and from a completely analogous expression for $\odu_{2}$, together with \eqref{eq:od_flux}, we arrive at the second condition on the $\oddens_{i}$,
  \begin{equation}
    \label{eq:od_mod_v}
    \frac{\oddens_{1}(t)}{\sqrt{\odfpc_{1}}}
    =
    \frac{\oddens_{2}(t)}{\sqrt{\odfpc_{2}}}
    \,.
  \end{equation}
  The system \eqref{eq:od_mod_ii}, \eqref{eq:od_mod_v} has the unique constant solution
  \begin{equation}
    \label{eq:od_mod_vi}
    \oddens_{i}(t)
    =
    \oddir\sqrt{\odfpc_{i}}
    \,,
    \qquad
    \oddir
    :=
    2\odinit
    (\sqrt{\odfpc_{2}}-\sqrt{\odfpc_{1}})
    \,,
    \quad
    i=1\,,2
    \,.
  \end{equation}
  Then \eqref{eq:od_mod_n} follows from \eqref{eq:od_mod_jump} and from \eqref{eq:od_mod_vi}.
\end{proof}

\begin{lemma}
  \label{l:od_zero}
  Let $\odu(x,t)=\odu_{i}(x,t)$ for $(-1)^{i}x>0$ be a 
solution to a problem obtained complementing \eqref{eq:od_pde}--\eqref{eq:od_flux} with the initial condition
  \begin{equation}
    \label{eq:od_zero_init}
    \odu(x,0)
    =
    \initd(x)
    \,,
    \qquad
    x\in\R
    \,,
  \end{equation}
  where $\initd$ is bounded in $\R$ with $\initd\in \mathcal{C}^{0}((-\sigma,\sigma))$ for some $\sigma>0$, and $\initd(0)=0$. Then
  $\odu$ is continuous at $(0,0)$ with value $\odu(0,0)=0$.
\end{lemma}

\begin{proof}
  Denote $\odfpc(x)=\odfpc_{i}$ if $(-1)^{i}x>0$. Fix $\eps>0$.
  We use the test function $(\odv-\eps)_{+}\zeta^{2}$,
 where we recall the definition of positive 
 part $(a)_{+}=\max\{a,0\}$, in the 
  weak formulation \eqref{eq:weak_form12} for $\odv=\odfpc\odu$, where
  \begin{equation*}
    \zeta\in \mathcal{C}^{1}(\R)
    \,,
    \quad
    \zeta(x)=1
    \,,
    \quad \abs{x}<\delta
    \,,
    \quad
    \zeta(x)=0
    \,,
    \quad
    \abs{x}>2\delta
    \,,
  \end{equation*}
  and $0<\delta<\sigma/4$ is such that $\odfpc\initd(x)<\eps/2$ for $\abs{x}<4\delta$. We get by routine calculations
  \begin{multline}
    \label{eq:od_zero_i}
    \int_{\R}
    \frac{1}{\odfpc(x)}
    (\odv(x,t)-\eps)_{+}^{2}\zeta(x)^{2}
    \di x
    \\
    \le
    \gamma
    \norma{\zeta'}{\infty}^{2}
    \int_{0}^{t}
    \int_{\delta<\abs{x}<2\delta}
    (\odv(x,\tau)-\eps)_{+}^{2}
    \di x
    \di \tau
    =
    0
    \,,
  \end{multline}
  where the last equality follows from our choice of $\delta$ for small enough $t>0$, when we take into account that, away from $x=0$, $\odv$ is as smooth as the data allow, since it solves a standard heat equation with constant diffusivity up to time $t=0$. Hence, in the region $\delta<\abs{x}<2\delta$, we have that for small times $\odv$ is close to its initial data $\odv(x,0)=\beta(x)\initd(x)<\eps/2$. In a similar way we prove $\odv\ge -\eps$ near $(0,0)$.

  Note that $\odv(x,0)$ is continuous in $(-\sigma,\sigma)$, in fact even at $x=0$; the present result might in fact follow from the theory of parabolic equations, but we prefer to give the above explicit proof because the weak formulation of the problem for $\odv$ is not completely standard.
\end{proof}

Our next result provides the counterexample to the sup bounds announced in Remark~\ref{r:supbound}.

\begin{proposition}
  \label{p:blow}
  Consider the problem
  \begin{alignat}{2}
    \label{eq:od_blow_n}
    \odu_{t}
    -
    (\odfpc\odu)_{xx}
    &=
    0
    \,,
    &\qquad&
    x\in\R
    \,,
    t>0
    \,,
    \\
    \label{eq:od_blow_nn}
    \odu(x,0)
    &=
    \odinit
    \,,
    &\qquad&
    x\in\R
    \,.
  \end{alignat}
  Here $\odinit>0$ is a given constant.

  Consider a sequence 
  $(x_{j},t_{j})\in (0,1)\times(0,1)$ with 
  $x_j,t_i$ increasing with $j$, 
  $(x_{j},t_{j})\to(\bar x,\bar t)$ 
  as
  $j\to+\infty$, $(x_{0},t_{0})=(0,0)$.
  For each $(x,t)\in\R\times(0,T)$ there exists a 
  unique $j$ such that $t\in(t_j,t_{j+1}]$. 
  We set
  \begin{equation}
    \label{eq:od_blow_s}
    \odfpc(x,t)
    =
    \left\{
      \begin{alignedat}{2}
        &16
          \,,
          &\qquad&
          x\le x_{j}
                   \,,
                   t_{j}<t\le t_{j+1}
                   \,,
                   \\
        &1
          \,,
          &\qquad&
          x> x_{j}
                   \,,
                   t_{j}<t\le t_{j+1}
                   \,.
      \end{alignedat}
      \right.
  \end{equation}
  Then there exists a solution $u$ to 
 \eqref{eq:od_blow_n}--\eqref{eq:od_blow_nn}
 such that
  \begin{equation}
    \label{eq:od_blow_m}
    \lim_{t\to \bar t-}
    \sup_{x\in(0,1)}\odu(x,t)
    =
    +\infty
    \,.
  \end{equation}
\end{proposition}
\begin{proof}
  The solution $\odu$ to \eqref{eq:od_blow_n}--\eqref{eq:od_blow_nn} will be constructed together with the sequence $(x_{j},t_{j})$, as the solution to initial value problems for equations of the type of \eqref{eq:od_pde}, each one valid in the time interval $(t_{j},t_{j+1})$. In the interval $t\in(t_{0},t_{1})$, $\odu$ coincides exactly with the solution to problem \eqref{eq:od_pde}--\eqref{eq:od_initd} with the choice
  \begin{equation}
    \label{eq:od_blow_jjj}
    \beta_1 
    =
    16
    \qquad
    \textup{ and }
    \qquad
    \beta_2
    =
    1
    \,,
  \end{equation}
  which corresponds to \eqref{eq:od_blow_s} with $x_{0}=t_{0}=0$ and $t_{1}$ to be chosen presently.
  Indeed, owing to Lemma~\ref{l:od_mod}, we may find
  \begin{equation}
    \label{eq:od_blow_i}
    0<x_{1}<\frac{1}{2}
    \,,
    \qquad
    0<t_{1}<\frac{1}{2}
    \,,
    \qquad
    \text{such that}
    \quad
    \odu(x_{1},t_{1}-)>2\alpha\,.
  \end{equation}
  For $t\in(t_{1},t_{2})$, $\odu$ is the solution to a new problem, with 
  $\odfpc(x,t)$ as in \eqref{eq:od_blow_s} (with $t_{2}$ to be chosen) and initial data $\initd^{1}(x)=\odu(x,t_{1}-)$. By linearity,  $\odu$ is given as $\odu=\tilde\odu^{1}+\hat\odu^{1}$ where
  \begin{alignat*}{2}
    \tilde\odu^{1}_{t}
    -
    (\odfpc\tilde\odu^{1})_{xx}
    &=
    0
    \,,
    &\qquad&
    x\in\R
    \,,
    t>t_{1}
    \,,
    \\
    \tilde\odu^{1}(x,t_{1})
    &=
    \odu(x_{1},t_{1}-)
    \,,
    &\qquad&
    x\in\R
    \,,
  \end{alignat*}
  and
  \begin{alignat*}{2}
    \hat\odu^{1}_{t}
    -
    (\odfpc\hat\odu^{1})_{xx}
    &=
    0
    \,,
    &\qquad&
    x\in\R
    \,,
    t>t_{1}
    \,,
    \\
    \hat\odu^{1}(x,t_{1})
    &=
    \odu(x,t_{1}-)
    -
    \odu(x_{1},t_{1}-)
    \,,
    &\qquad&
    x\in\R
    \,.
  \end{alignat*}
  Note that we may apply Lemma~\ref{l:od_mod} to $\tilde\odu^{1}$ 
  with $\alpha$ replaced by $u(x_1,t_1-)$ to get
  \begin{equation}
    \label{eq:od_blow_j}
    \tilde\odu^{1}(x_{1}+,t)
    =
    4\odu(x_{1},t_{1}-)
    >
    4\cdot2\odinit
    =
    8\odinit
    \,,
    \qquad
    t>t_{1}
    \,,
  \end{equation}
  while, owing to Lemma~\ref{l:od_zero}, $\hat\odu^{1}$ is continuous at $(x_{1},t_{1})$, with zero value. Thus it is possible to find
  \begin{equation}
    \label{eq:od_blow_ii}
    x_{1}<x_{2}<x_{1}+\frac{1}{4}<\frac{1}{2}+\frac{1}{4}
    \,,
    \quad
    t_{1}
    <
    t_{2}
    <
    t_{1}+\frac{1}{4}
    <
    \frac{1}{2}+\frac{1}{4}
    \,,
  \end{equation}
  such that
  \begin{equation}
    \label{eq:od_blow_iii}
    \odu(x_{2},t_{2}-)
    >
    4\odinit
    \,.
  \end{equation}
  Proceeding by induction we find increasing sequences $x_{j}$, $t_{j}$ such that
  \begin{equation}
    \label{eq:od_blow_k}
    0<x_{j}<\sum_{i=1}^{j} 2^{-i}
    \,,
    \qquad
    0<t_{j}<\sum_{i=1}^{j} 2^{-i}
    \,,
    \qquad
    \odu(x_{j},t_{j}-)>2^{j}\odinit
    \,.
  \end{equation}
  Note that the construction above is logically consistent, since the problem for $t<t_{j}$ does not depend on the choice of $\odfpc$ for $t>t_{j}$. Also note that the limit $\odu(x_{j},t_{j}-)$ is taken in the classical sense (though actually $\odu$ is not continuous at $(x_{j},t_{j})$, for $t\to t_j+$).
  The proof is concluded.
\end{proof}

It is easily seen from the proof that $(\bar x,\bar t)$ might in fact be chosen as close to $(0,0)$ as wanted. More importantly, we remark that at least in the present case where the dependence of $\odfpc$ on $t$ is piecewise constant, a local uniform $L^{1}$ bound for the solution in the spirit of Lemma~\ref{l:mass_con} can be proved following the same ideas.

\section{Conclusions}
\label{s:concl}

We have considered a Fokker--Planck diffusion equation
for an inhomogeneous material with inclusions 
of size $\eta\eps$ in which the magnitude of the diffusion 
coefficient is controlled by the parameter $\delta$.
We assumed a periodic microstructure of period $\eps$ 
and have derived the upscaled equations taking 
the degeneration $\delta\to0$ and the homogenization $\eps\to0$ 
limits under a set of exhaustive assumptions on $\eta$.

In the Introduction, see Section~\ref{s:intro}, we have 
described in detail our results and discussed both their 
mathematical and physical meaning with the specific references
to the theorems proven in the paper. 
In this conclusive section we summarize 
these results 
in Table~\ref{tab:conc}.

\begin{table}
\begin{tblr}{
  columns={colsep=2.5pt},
  colspec = {Q[c,m]Q[c,m]Q[c,m]Q[c,m]Q[c,m]},
  cell{2}{2} = {c=4}{c}, 
  cell{3}{4} = {c=2}{c}, 
  cell{4}{2} = {c=3}{c}, 
  cell{5}{2} = {c=3}{c}, 
  rowsep = 6pt,
  hlines = {black, 0.8pt},
  vlines = {black, 0.8pt}, 
}

 & $\pmb{\eta\ll\eps^{2/(n-2)}}$ & 
   $\eta\approx\eps^{2/(n-2)}$ & 
   $\newatop{\eta\gg\eps^{2/(n-2)}}{\eta(\eps)\to0}$ & 
   $\pmb{\eta=1}$ \\
$\newatop{\eta,\eps>0}{\delta\to0}$ &
          $\newatop{\textrm{\eqref{eq:pde_2}--\eqref{eq:init_2}--PD 
                            outside the inclusions}}
                   {\textrm{\eqref{eq:pde_1}--\eqref{eq:init_1}--AD
                             inside the inclusions}}$
                      \\
$\newatop{\delta=0}{\eta,\eps\to0}$ 
 & $\newatop{\textrm{Theorem~\ref{t:t2_supcrit}}}
            {\textrm{\eqref{eq:m26ter}--PD}}$ 
 & $\newatop{\textrm{Theorem~\ref{t:t2}}}
            {\textrm{\eqref{eq:m26bis}--DMD}}$ 
 & $\newatop{\textrm{Theorem~\ref{t:superd}}}
            {\textrm{\eqref{eq:superd_nn}--AD}}$ 
 & \\
$\newatop{\delta>0}{\eta,\eps\to0}$ 
 & $\newatop{\textrm{Theorem~\ref{t:t22}}}
            {\textrm{\eqref{eq:m26ter}--PD}}$ 
 & &  
 & $\newatop{\textrm{Theorem~\ref{t:t23}}}
            {\textrm{\eqref{eq:a46}--PD}}$ 
 \\
$\newatop{\eta,\eps=0}{\delta\to0}$ 
 & $\newatop{\textrm{Theorem~\ref{t:t22}}}
            {\textrm{\eqref{eq:m26ter}--PD}}$ 
 & &  
 & $\newatop{\textrm{Theorem~\ref{t:t10}}}
            {\textrm{\eqref{eq:superd_nn}--AD}}$ 
 \\
\end{tblr}
\vskip 0.5 cm
\caption{Summary of the results: see the text for the detailed description 
of the table entries. Boldface characters denote cases in which the first 
and the second asymptotic schemes commute. }
\label{tab:conc} 
\end{table}

The upscaled problems that we have found in the different 
cases that we have analyzed can be classified 
as pure diffusion,
diffusion with mass deposition, and absence of diffusion.
In the table we use, respectively, the acronyms 
(PD), (DMD), and (AD) to refer to them. 

The four rows in the table refer to the different limits 
that we have considered:
$\eta,\eps>0$, $\delta\to0$ refers to the degeneration limit 
$\delta\to0$ taken for fixed $\eta$ and $\eps$;
$\delta=0$, $\eta,\eps\to0$, refers to the homogenization 
limit of the degenerated problem;
$\delta>0$, $\eta,\eps\to0$, refers to the homogenization 
limit for fixed diffusion magnitude $\delta$;
$\eta,\eps=0$, $\delta\to0$ refers to the degeneration limit 
$\delta\to0$ of the previously homogenized problem. 

The four columns refer to the four different 
exhaustive cases that we have considered for the 
dependence of $\eta$ on $\eps$ when the homogenization limit 
$\eps\to0$ has been computed. We have addressed them 
as the 
subcritical, the critical, the supercritical, and the constant cases, 
with the last 
one being a special sub-case of the supercritical case. 
Note that, depending on the 
specific row, some of the columns are merged since they share 
the same result. 

Finally, table entries of the first row are in boldface font 
in the cases in which the results in the second and in the 
fourth rows are equal. Indeed, in these cases 
the order in which the degeneration and the homogenization limits
are taken  
is not relevant, that is to say, the two asymptotic schemes 
discussed in the Introduction commute. 



\end{document}